\documentclass[10pt,twoside]{amsart}
\usepackage{amsmath}
\usepackage{amssymb}
\usepackage{amsthm}
\usepackage{latexsym}
\usepackage{amscd}
\usepackage{xypic}
\xyoption{curve}
\usepackage{ifthen}
\usepackage{hyperref}
\usepackage{graphicx}

\usepackage{xcolor}

\newtheorem{theorem}{Theorem}[section]
\newtheorem{lemma}[theorem]{Lemma}
\newtheorem{proposition}[theorem]{Proposition}
\newtheorem{corollary}[theorem]{Corollary}
\newtheorem{remark}[theorem]{Remark}
\newtheorem{claim}[theorem]{Claim}
\theoremstyle{definition}
\newtheorem{definition}[theorem]{Definition}
\newtheorem{example}[theorem]{Example}

\newcommand {\Ext}{\mathrm{Ext}}
\newcommand {\im}{\mathrm{im}}
\newcommand {\rk}{\mathrm{rk}}

\newcommand {\Hilb}{\mathcal{H}\kern -0.25ex{\mathit ilb\/}}
\newcommand {\PGL}{\mathrm{PGL}}

\newcommand {\bZ}{\mathbb{Z}}

\newcommand {\bP}{\mathbb{P}}

\newcommand{\cC}{{\mathcal C}}
\newcommand{\cE}{{\mathcal E}}
\newcommand{\cF}{{\mathcal F}}
\newcommand{\cG}{{\mathcal G}}
\newcommand{\cM}{{\mathcal M}}
\newcommand{\cN}{{\mathcal N}}
\newcommand{\cO}{{\mathcal O}}

\newcommand{\cI}{{\mathcal I}}
\newcommand{\GL}{\operatorname{GL}}

\newcommand{\Pic}{\operatorname{Pic}}

\def\p#1{{\bP^{#1}}}

\def\ga#1{{{\accent"12 #1}}}

\title[rank two aCM bundles]{Rank two aCM bundles \\
on the del Pezzo fourfold of degree $6$
\\and its general hyperplane section}



\author[G. Casnati, D. Faenzi, F. Malaspina]{Gianfranco Casnati, Daniele Faenzi, Francesco Malaspina}
\thanks{All the authors are members of GRIFGA--GDRE project, supported by CNRS
  and INdAM, and of the GNSAGA group of INdAM. The first and third authors are supported by the framework of PRIN 2010/11 \lq Geometria delle variet\ga a algebriche\rq, cofinanced by MIUR. The second author is partially supported by ANR GEOLMI contract ANR-11-BS03-0011}

\date{\today}

\begin{document}

\begin{abstract}
In the present paper we completely classify locally free sheaves of rank $2$ with vanishing  intermediate cohomology modules on the image of the Segre embedding $\p2\times\p2\subseteq\p8$ and its general hyperplane sections.

Such a classification extends similar  already known results  regarding del Pezzo varieties with Picard numbers $1$ and $3$ and dimension at least $3$.
\end{abstract}

\maketitle

\section{Introduction}
Let  $\p N$ be the $N$-dimensional projective space over an
algebraically closed field $k$ of characteristic $0$. If $X\subseteq\p
N$ is a smooth closed submanifold of dimension $n \ge 1$ we set $\cO_X(1):=\cO_{\p N}(1)\otimes\cO_X$.

Among vector bundles
 $\cF$ on such an $X$, the simplest ones from the cohomological point of view satisfy the vanishing
$$
H^i\big(X,\cF(t)\big)=0, \qquad \forall\ i=1,\dots,n-1,\ t\in\bZ.   
$$
Such vector bundles are called arithmetically Cohen-Macaulay (aCM for
short). 

It is possible to  classify aCM
bundles completely for varieties belonging to a short
precise list, consisting (for $n \ge 2$) of
projective spaces, 
smooth quadrics, the Veronese surface in $\p5$ and rational normal surface scrolls of degree up to
$4$ (see \cite{F--M} for quartic scrolls and \cite{Ho}, \cite{Kno},
\cite{A--R}, \cite{Sol} for the
remaining cases: cf. also \cite{E--H}).
However, for any other variety such a classification is hopeless, due
to the \lq\lq wild\rq\rq\  behaviour of the category of maximal Cohen--Macaulay modules (which are the algebraic counterpart of aCM bundles), as
described in \cite{F--PL}. 

In spite of this, a detailed study of aCM bundles of low rank, in
particular of rank $2$, is feasible for certain types of varieties, most
notably Fano manifolds. These are smooth projective varieties $X$ of dimension $n$, such that $\omega_X^{-1}$ is ample (see \cite{I--P} for a review about Fano varieties). Let $r$ be the  greatest positive integer such that
$\omega_X\cong\mathcal L^{-r}$ for some ample $\mathcal L\in\Pic(X)$. It is known that $1\le r\le n+1$ and
$r=n+1$ (resp. $r=n$) if and only if $X=\p n$ (resp. $X$ is a
 quadric hypersurface).  

We are thus interested in the case $r\le n-1$. Fano manifolds satisfying the equality $r=n-1$ are called del Pezzo. They are
completely classified and fall into a very short list (see \cite{Fu}, 
\cite{I--P}).

We focus our attention to the case $n\ge 3$: for simplicity we assume $\mathcal L$ to be very ample. Several results about aCM bundles on Fano and del Pezzo
threefolds can be found in  \cite{A--C}, \cite{A--G}, \cite{B--F1},
\cite{Ca}, \cite{Ma3}, \cite{Fa1}, \cite{A--F}, when the Picard number $\varrho(X):=\rk\Pic(X)$ of $X$ equals $1$.
Let us point out that aCM bundles of rank $2$ on our variety $X$ are
intimately related to several important features of $X$, as for instance
for a cubic hypersurface $X$ these bundles are in bijection with inequivalent Pfaffian representations
of $X$. 
In a different direction, stable aCM bundles $\cF$ of rank $2$ with
$c_1(\cF)=0$ on a del Pezzo threefold $X$ correspond
to instantons of minimal charge, cf. \cite{Fa2}, \cite{Kuz} which appear in connection
with the intermediate Jacobian of $X$, and with periods of other classes of
Fano manifolds, \cite{M--T}.
Anyway these bundles have  been
studied mainly under the assumption $\varrho(X)= 1$, while technical
difficulties arise for
higher values of $\varrho(X)$,
mainly because the
correspondence between a rank $2$ bundle $\cF$ and the zero locus of a global section of
$\cF$ is unclear a priori: this is indeed the starting point
of our analysis.

The first complete classification of aCM bundles of rank $2$ on del
Pezzo manifolds $X$ with
$\varrho(X) \ge 2$ is presented in \cite{C--F--M1} for the Segre embedding
$\p1\times\p1\times\p1\subseteq\p7$; the moduli spaces of these
bundles are
studied  in \cite{C--F--M2}. In this case
$\deg(X)=6$ and $\varrho(X)$ is as high as possible.

The aim of the present paper is to give a complete description of
indecomposable aCM bundles of rank $2$ on the other del Pezzo
manifolds of degree $6$, namely the Segre embedding
$\Phi:=\p2\times\p2\subseteq\p8$ and its general hyperplane section
$F$; of course both $F$ and $\Phi$ have Picard number $2$.
In order to formulate our main results, we introduce some notation. Let $A(X)$ be the Chow ring of $X$, so that $A^r(X)$ denotes the set of  cycles of codimension $r$. 
The two projections $\pi_i\colon \Phi\to\p2$, $i=1,2$, are isomorphic
to the canonical map $\bP(\cO_{\p2}^{\oplus3})\to\p2$. We denote by
$\eta_i$, $i=1,2$, the classes of $\pi_i^*\cO_{\p2}(1)$ in
$A^1(\Phi)$, so that $\Pic(\Phi)\cong\bZ^{\oplus2}$ is
generated by $\eta_1$ and $\eta_2$. Clearly,  $\eta_2$ and $\eta_1$ are respectively identified with the tautological divisors of $\pi_1$ and $\pi_2$, and the class of the
hyperplane divisor on $\Phi$ is $\eta=\eta_1+\eta_2$. Then:
$$
A(\Phi)\cong A(\p2)\otimes A(\p2)\cong\bZ[\eta_1,\eta_2]/(\eta_1^3,\eta_2^3).
$$

The morphisms $\pi_i$ induce maps $p_i\colon F\to\p2$ by restriction, $i=1,2$. Such maps are isomorphic to the canonical map $\bP(\Omega_{\p2}^1(2))\to\p2$. Thinking of the second copy of $\p2$ as the dual of the first one, then $F$ can also be viewed naturally as the flag variety of pairs point--line in $\p2$. Let  $h_i$, $i=1,2$, be the classes of $p_i^*\cO_{\p2}(1)$ in $A^1(F)$ respectively. As for $\Phi$ we have that $h_2$ and $h_1$ are respectively identified with the tautological divisors of $p_1$ and $p_2$, so that the class of the hyperplane divisor on $F$ is $h=h_1+h_2$. The above discussion proves the isomorphisms
$$
A(F)\cong A(\p2)[h_1]/(h_1^2-h_1h_2+h_2^2)\cong \bZ[h_1,h_2]/(h_1^2-h_1h_2+h_2^2, h_1^3,h_2^3).
$$
In particular, $\Pic(F)\cong\bZ^{\oplus2}$ with generators $h_1$ and $h_2$. 

In Section \ref{saCM}, after recalling Hartshorne--Serre correspondence, we give a characterization of aCM line bundles on $F$. In the final part of the section we show that initialized, indecomposable, aCM bundles of rank $2$ on $\Phi$ restrict on $F$ to bundles with the same properties. In particular the descriptions of bundles on $F$ and $\Phi$ are strictly interlaced. We will use such a relationship in the whole paper: indeed, usually, we will first deal with one of the del Pezzo sextics, then using the obtained results for dealing with the other one.

Our first main results are the following (see Sections \ref{sLowerChern6} and \ref{sUpperChern6}).

\medbreak
\noindent
{\bf Theorem A for $F$.}
{\it Let $\cE$ be an indecomposable aCM bundle of rank $2$ on $F$ and let $c_1:=c_1(\cE)=\alpha_1h_1+\alpha_2h_2$. 
Assume that $h^0\big(F,\cE\big)\ne0$ and $h^0\big(F,\cE(-h)\big)=0$.  Then:
  \begin{enumerate}
  \item the zero locus $E:=(s)_0$ of a general section $s\in H^0\big(F,\cE\big)$ has pure codimension $2$ in $F$;
  \item $0\le\alpha_i\le 2$, $i=1,2$.
  \end{enumerate}}
\medbreak

\medbreak
\noindent
{\bf Theorem A for $\Phi$.}
{\it Let $\cG$ be an indecomposable aCM bundle of rank $2$ on $\Phi$ and let $\gamma_1:=c_1(\cG)=\alpha_1\eta_1+\alpha_2\eta_2$. 
Assume that $h^0\big(\Phi,\cG\big)\ne0$ and $h^0\big(\Phi,\cG(-\eta)\big)=0$.  Then:
  \begin{enumerate}
  \item the zero locus $\Sigma:=(\sigma)_0$ of a general section $\sigma\in H^0\big(\Phi,\cG\big)$ has pure codimension $2$ in $\Phi$;
  \item $0\le\alpha_i\le 2$, $i=1,2$.
  \end{enumerate}}
\medbreak

In Section \ref{sExtremal} we provide a complete classification of
bundles on $F$ and $\Phi$ with first Chern class $c_1=0,2h$, $\gamma_1=0,2\eta$ respectively. Finally,  in Section \ref{sIntermediate}, we will determine which intermediate cases are actually
admissible. 

We summarize the second main results in the following simplified statements (for expanded and detailed statements on vector bundles on $F$ and $\Phi$ see Theorems \ref{tline}, \ref{tElliptic},  \ref{tIntermediate3}  and \ref{tplane}, \ref{tdelPezzo}, \ref{tIntermediate4} respectively).

\medbreak
\noindent
{\bf Theorem B for $F$.}
{\it There exists an indecomposable and initalized aCM bundle $\cE$ of rank $2$ on
  $F$ with $h^0\big(F,\cE\big)\ne0$, $h^0\big(F,\cE(-h)\big)=0$ and 
$c_1=\alpha_1h_1+\alpha_2h_2$ if and only if  $(\alpha_1,\alpha_2)$ is one of the following, up to permutations:
$$
(0,0),\quad (0,1),\quad (1,2),\quad (2,2).
$$
Moreover, denote by $E$ the zero--locus of a general section of such an $\cE$. Then
  \begin{enumerate}
  \item if $\alpha_1=\alpha_2=0$, then $E$ is a line and each  line on $F$ can be obtained in such a way;
 \item if $\alpha_1=0$, $\alpha_2=1$, then $E$ is a line, $\cE\cong p_2^*\Omega_{\p2}^1(2h_2)$  and  each line on $F$ can be obtained in such a way;
  \item if $\alpha_1=1$, $\alpha_2=2$, then $E$ is a, possibly reducible, quartic with arithmetic genus $0$ and $\cE\cong p_1^*\Omega_{\p2}^1(2h_1+h_2)$;
 \item if $\alpha_1=\alpha_2=2$, then $E$ is a smooth elliptic normal curve of degree $8$ in $F\subseteq\p7$ and each such a curve in $F$ can be obtained in this way.
  \end{enumerate}}
\medbreak

\medbreak
\noindent
{\bf Theorem B for $\Phi$.}
{\it There exists an indecomposable and initalized aCM bundle $\cG$ of rank $2$ on
  $\Phi$ with $h^0\big(\Phi,\cG\big)\ne0$, $h^0\big(\Phi,\cG(-\eta)\big)=0$ and
$\gamma_1=\alpha_1\eta_1+\alpha_2\eta_2$ if and only if  $(\alpha_1,\alpha_2)$ is one of the following, up to permutations:
$$
(0,0),\quad (0,1),\quad (1,2),\quad (2,2).
$$
Moreover, denote by $\Sigma$ the zero--locus of a general section of such an $\cG$. Then
  \begin{enumerate}
  \item if $\alpha_1=\alpha_2=0$, then $\Sigma$ is a plane and each  plane on $\Phi$ can be obtained in such a way;
 \item if $\alpha_1=0$, $\alpha_2=1$, then $\Sigma$ is a plane, $\cG\cong \pi_2^*\Omega_{\p2}^1(2\eta_2)$ and each  plane on $\Phi$ can be obtained in such a way;
  \item if $\alpha_1=1$, $\alpha_2=2$, then $\Sigma$ is a, possibly reducible, quartic surface and $\cG\cong \pi_1^*\Omega_{\p2}^1(2\eta_1+\eta_2)$;
 \item if $\alpha_1=\alpha_2=2$, then $\Sigma$ is a smooth del Pezzo surface of degree $8$ in $\Phi\subseteq\p8$ isomorphic to the blow up of $\p2$ at a point and each such a surface in $\Phi$ can be obtained in this way.
  \end{enumerate}}
\medbreak

The study of (semi)stability of aCM bundles of rank $2$ on $F$ and $\Phi$ and of the moduli spaces of such (semi)stable bundles will be  the object of a future paper.

Throughout the whole paper we refer to \cite{I--P} and \cite{Ha2} for all the unmentioned definitions, notations and results.

\section{Some facts on aCM locally free sheaves}
\label{saCM}
Throughout the whole paper $k$ will denote an algebraically closed
field of characteristic $0$. Let $X\subseteq \p N$ be a subvariety,
i.e. an integral closed subscheme defined over $k$, of dimension $n
\ge 1$. We set $\cO_X(1):=\cO_{\p N}(1)\otimes\cO_X$. We start this section by recalling two important definitions. 

The variety $X\subseteq \p N$ is called aCM if  the restriction maps $H^0\big(\p N,\cO_{\p N}(t)\big)\to H^0\big(X,\cO_{X}(t)\big)$ are surjective and $h^i\big(X,\cO_{X}(t)\big)=0$, $1\le i\le n-1$. 

The variety $X\subseteq \p N$ is called arithmetically Gorenstein (aG for short) if  it is aCM and $\alpha$--subcanonical, i.e. its dualizing sheaf satisfies $\omega_X\cong\cO_X(\alpha h)$ for some $\alpha\in \bZ$.

In what follows  $X$ will be an aCM, integral and smooth
subvariety of $\p N$ of dimension $n \ge 1$. 

\begin{definition}
  Let $\cF$ be a vector bundle on $X$.
 We say that $\cF$ is {\sl arithmetically Cohen--Macaulay}
  (aCM for short) if  the modules $H^i_*\big(X,\cF\big) = 
\bigoplus_{t\in\bZ}H^i\big(X,\cF(t)\big)
$ vanish for all $i=1,\dots,n-1$.
\end{definition}

If $\cF$ is an aCM bundle, then the minimal number of
generators $m(\cF)$ of $H^0_*\big(X,\cF\big)$ as a module over the
graded coordinate ring of $X$ is $\rk(\cF)\deg(X)$ at
most (e.g. see \cite{C--H1}). The aCM bundles for which the maximum is
attained are  called Ulrich
bundles (see \cite{C--H2}, Definition 2.1 and Lemma 2.2: see also \cite{C--H1}, Definition 3.4
which is slightly different).

\begin{definition}
  Let $\cF$ be a vector bundle on $X$.
  We say that $\cF$ is {\sl initialized} if 
  $$
  \min\{\ t\in\bZ\ \vert\ h^0\big(X,\cF(t)\big)\ne0\ \}=0.
  $$
  We say that $\cF$ is {\sl Ulrich} if it is initialized, aCM and
  $h^0\big(X,\cF\big)=\rk(\cF)\deg(X)$. 
\end{definition}

Let $\cF$ be an Ulrich bundle. On the one hand we know that
$m(\cF)=\rk(\cF)\deg(X)$. On the other hand the generators of
$H^0\big(X,\cF\big)$ are minimal generators of $H^0_*\big(X,\cF\big)$
due to the vanishing of $H^0\big(X,\cF(-1)\big)$. We conclude that
$\cF$ is necessarily globally generated. Several other results are
known for Ulrich bundles (e.g. see \cite{E--S}, \cite{C--H1}, \cite{C--H2},
\cite{C--K--M}).

Let $f$ be a global section of a rank $2$ vector bundle $\cF$. In general its zero--locus
$(f)_0\subseteq X$ is either empty or its codimension is at most
$2$. In this second case, we can always write $(f)_0=C\cup \Delta$
where $C$ has pure codimension $2$ (or it is empty) and $\Delta$ has pure codimension
$1$ (or it is empty). In particular $\cF(-\Delta)$ has a section ${f_\Delta}$ vanishing
on $C$, which is thus locally complete intersection inside $X$. Moreover, the Koszul complex of ${f_\Delta}$ twisted by $\cO_X(\Delta)$ is
\begin{equation}
  \label{seqIdeal}
  0\longrightarrow \cO_X(\Delta)\longrightarrow \cF\longrightarrow \cI_{C\vert X}(c_1(\cF)-\Delta)\longrightarrow 0.
\end{equation}
Moreover we also have the following exact sequence
\begin{equation}
  \label{seqStandard}
  0\longrightarrow \cI_{C\vert X}\longrightarrow \cO_X\longrightarrow \cO_C\longrightarrow 0.
\end{equation}

The above construction can be reversed, because Hartshorne--Serre correspondence holds (for further details see \cite{Vo}, \cite{Ha1}, \cite{Ar}).

\begin{theorem}
  \label{tSerre}
  Let $C\subseteq X$ be a local complete intersection subscheme of codimension $2$ and $\mathcal L$ an invertible sheaf on $X$ such that $H^2\big(X,{\mathcal L}^\vee\big)=0$. If $\det(\cN_{C\vert X})\cong\mathcal L\otimes\cO_C$, then there exists a vector bundle $\cF$ of rank $2$ on $X$ such that:
  \begin{enumerate}
  \item $\det(\cF)\cong\mathcal L$;
  \item $\cF$ has a section $f$ such that $C$ coincides with the zero locus $(f)_0$ of $f$.
  \end{enumerate}
  Moreover, if $H^1\big(X,{\mathcal L}^\vee\big)= 0$, the two conditions above determine $\cF$ up to isomorphism.
\end{theorem}

From now on we will focus our attention on the del Pezzo fourfold $\Phi\subseteq\p8$ of degree $6$ and its general hyperplane section $F\subseteq\p7$ which are both aCM varieties. The aim of the remaining part of to inspect the relations among aCM bundles of ranks $1$ and $2$ on $F$ and $\Phi$.

As pointed out in the Introduction
$$
A(\Phi)\cong \bZ[\eta_1,\eta_2]/(\eta_1^3,\eta_2^3),\qquad A(F)\cong \bZ[h_1,h_2]/(h_1^2-h_1h_2+h_2^2, h_1^3,h_2^3).
$$

We recall the following representation of $F$. Let $B\subseteq \GL_3$ be the subgroup of upper triangular matrices. The group $B$ is mapped by the natural morphism $\GL_3\to\PGL_3$ onto the subgroup fixing both the point $[1,0,0]$ and the line $\{\ x_2=0\ \}$. It follows the existence of a natural isomorphism $F\cong \GL_3/B$. Such a representation of $F$ is particularly helpful, because it allows us to state the following consequence of Borel--Weil--Bott Theorem (see \cite{De},Theorem 2) in order to compute the cohomology of the line bundles on $F$. 

\begin{theorem}
\label{tBott}
If $\mathcal L\in\Pic(F)$, then $h^i\big(F,\mathcal L\big)\ne0$ for at most one $i=0,\dots,3$.
\end{theorem}

We have the following standard restriction exact sequence
\begin{equation*}
0\longrightarrow \cO_{\Phi}(\alpha_1\eta_1+\alpha_2\eta_2-\eta)\longrightarrow \cO_{\Phi}(\alpha_1\eta_1+\alpha_2\eta_2)\longrightarrow \cO_{F}(\alpha_1h_1+\alpha_2h_2)\longrightarrow0
\end{equation*}
for each $ \cO_{F}(\alpha_1h_1+\alpha_2h_2)\in\Pic(F)$. Since the cohomology of $\cO_{\Phi}(\alpha_1\eta_1+\alpha_2\eta_2)$ vanishes in odd dimension thanks to K\"unneth formulas, it follows that
\begin{equation*}
\begin{aligned}
h^1\big(F,\cO_F(\alpha_1h_1+\alpha_2h_2)\big)-h^2\big(F,\cO_F(\alpha_1h_1+\alpha_2h_2)\big)&=\\
=h^2\big(\Phi,\cO_{\Phi}(\alpha_1\eta_1+\alpha_2\eta_2-\eta\big)-h^2&\big(\Phi,\cO_{\Phi}(\alpha_1\eta_1+\alpha_2\eta_2)\big).
\end{aligned}
\end{equation*}
Thanks to Theorem \ref{tBott}, only one among $h^1\big(F,\cO_F(\alpha_1h_1+\alpha_2h_2)\big)$ and $h^2\big(F,\cO_F(\alpha_1h_1+\alpha_2h_2)\big)$ in non--zero, according to the sign of the second member of the equality above.

More precisely we have the following Proposition.

\begin{proposition}
\label{pLineBundle}
For each $\alpha_1,\alpha_2\in\bZ$ with $\alpha_1\le \alpha_2$, we have $$
h^i\big(F,\cO_F(\alpha_1h_1+\alpha_2h_2)\big)\ne0
$$
if and only if
\begin{itemize}
\item $i=0$ and $\alpha_1\ge0$;
\item $i=1$ and $\alpha_1\le -2$, $\alpha_1+\alpha_2+1\ge0$;
\item $i=2$ and $\alpha_2\ge0$, $\alpha_1+\alpha_2+3\le0$;
\item $i=3$ and $\alpha_2\le -2$.
\end{itemize}
In all these cases
$$
h^i\big(F,\cO_F(\alpha_1h_1+\alpha_2h_2)\big)=(-1)^i\frac{(\alpha_1+1)(\alpha_2+1)(\alpha_1+\alpha_2+2)}{2}.
$$
\end{proposition}

An helpful tool for summarizing the first part of the above proposition is the following picture dealing with the subsets of the plane $\alpha_1\alpha_2$ whose points correspond to some non--zero cohomology group of the sheaf $\cO_F(\alpha_1h_1+\alpha_2h_2)$ (without restriction on $\alpha_1$ and $\alpha_2$).

\centerline{
\setlength{\unitlength}{0.5cm}
\begin{picture}(20,23)(-10,-12)
\multiput(-9,0)(1,0){18}{\line(1,0){0.5}}
\put(9,0){\vector(1,0){0.5}}
\put(8.7,0.3){\tiny$x_1$}
\multiput(0,-9)(0,1){18}{\line(0,1){0.5}}
\put(0,9){\vector(0,1){0.5}}
\put(0.3,9){\tiny$x_2$}
\put(-2,1){\line(-1,1){6}}
\put(0,-3){\line(1,-1){6}}
\put(-7.8,7){\tiny$x_1+x_2+1=0$}
\put(-2,1){\line(0,1){8}}
\put(0,-3){\line(0,-1){6}}
\put(-4.4,-8){\tiny$x_1=-2$}
\put(-4.7,4){$h^1\ne0$}
\put(0,0){\line(0,1){9}}
\put(0,0){\line(1,0){9}}
\put(3,2.5){$h^0\ne0$}
\put(-2,-2){\line(0,-1){7}}
\put(-8,-2.6){\tiny$x_2=-2$}
\put(-2,-2){\line(-1,0){7}}
\put(-7,-5.5){$h^3\ne0$}
\put(-3,0){\line(-1,1){6}}
\put(1,-2){\line(1,-1){6}}
\put(1.3,-9){\tiny$x_1+x_2+3=0$}
\put(-3,0){\line(-1,0){6}}
\put(1,-2){\line(1,0){8}}
\put(0.5,-6.7){$h^1\ne0$}
\put(-7,0.5){$h^2\ne0$}
\put(3,-3.5){$h^2\ne0$}
\put(-1.5,-11){\rm{Figure 1}}
\end{picture}}

The points corresponding to the direct summands of $H^i_*\big(F,\cO_F(\alpha_1h_1+\alpha_2h_2)\big)$ lie on the line through $(\alpha_1,\alpha_2)$ and parallel to $x_1-x_2=0$, hence the following two corollaries are immediate.

\begin{corollary}
\label{cNonVanishing}
If $\alpha_1,\alpha_2\in\bZ$, then
$$
H^1_*\big(F,\cO_F(\alpha_1h_1+\alpha_2h_2)\big)\ne0\quad\Leftrightarrow\quad H^2_*\big(F,\cO_F(\alpha_1h_1+\alpha_2h_2)\big)\ne0.
$$
\end{corollary}
\begin{proof}
Looking at Figure 1, one immediately checks that the line $x_2=x_1+c$ intersects the subset $h^1\ne0$ if and only if it intersects the subset $h^2\ne0$.
\end{proof}

The following second important corollary shows that initialized aCM (resp. Ulrich) line bundles on $F$ are exactly the restrictions of initialized aCM (resp. Ulrich) line bundles on $\Phi$.

\begin{corollary}
\label{cLineBundle}
If $\alpha_1,\alpha_2\in\bZ$, then $\cO_F(\alpha_1h_1+\alpha_2h_2)$ is initialized and aCM if and only if $(\alpha_1,\alpha_2)$ is one of the following:
$$
(0,0),\quad(0,1),\quad(1,0),\quad(0,2),\quad(2,0).
$$
Moreover, $\cO_F(\alpha_1h_1+\alpha_2h_2)$ is Ulrich if and only if
$(\alpha_1,\alpha_2)$ is one of the following:
$$
(0,2),\quad(2,0).
$$
\end{corollary}
\begin{proof}
On the one hand the sheaf $\cO_F(\alpha_1h_1+\alpha_2h_2)$ is initialized if and only if $\alpha_1$ and $\alpha_2$ are non--negative and $\alpha_1\alpha_2=0$.
On the other hand, looking at Figure 1, one checks that the line $x_2=x_1+c$ does not intersect the subsets $h^1\ne0$ and $h^2\ne0$ if and only if $-2\le c\le2$.
Thus the first part of the statement follows.

The second part of the statement is an easy direct computation.
\end{proof}

The following remark will be used in the paper.

\begin{remark}
\label{rRestriction}
Let $\cG$ be an aCM bundle of rank $2$ on $\Phi$ and let $\cE:=\cG\otimes \cO_F$ be its restriction. By computing the cohomology of sequence
\begin{equation}
\label{seqRestriction}
0\longrightarrow \cG(-\eta)\longrightarrow \cG\longrightarrow \cE\longrightarrow 0
\end{equation}
twisted by $\cO_\Phi(t\eta)$ one checks that $\cE$ is aCM too.

If $\cG$ is initialized it is immediate to check that the same is true for $\cE$. Moreover the restriction morphism induced by the inclusion $F\subseteq\Phi$ yields an isomorphism $H^0\big(\Phi,\cG\big)\cong H^0\big(F,\cE\big)$. In this isomorphism the zero loci of the corresponding sections $\sigma\in H^0\big(\Phi,\cG\big)$ and $s\in H^0\big(F,\cE\big)$ satisfy $(s)_0=F\cap (\sigma)_0$.

Conversely if $\cE$ is initialized, the cohomology of Sequence \eqref{seqRestriction} twisted by $\cO_\Phi(t\eta)$ yields that $h^0\big(\Phi,\cG(t\eta)\big)=h^0\big(\Phi,\cG(-\eta)\big)$ when $t\le -1$. We conclude that $h^0\big(\Phi,\cG(-\eta)\big)=0$. Moreover $h^0\big(\Phi,\cG\big)=h^0\big(F,\cE\big)\ne0$, because $\cE$ is assumed to be initialized. It follows that also $\cG$ is initialized.

Finally it also follows that $\cG$ is Ulrich if and only if the same holds for $\cE$, because $\deg(\Phi)=\deg(F)$ and $\rk(\cG)=\rk(E)$.
\end{remark}

Let $\cG$ be an initialized, aCM bundle of rank $2$ on $\Phi$. Thanks to the above remark, we know that $\cE:=\cG\otimes\cO_F$ is initialized and aCM too. We will prove in the last part of this section that if $\cE$ is decomposable, then the same is true for $\cG$. To this purpose assume that $\cE$ splits as a sum of line bundles. Such line bundles are necessarily aCM and they are either both initialized or one of them is initialized and the other one has no sections. In particular we can assume that
$$
\cE\cong\cO_F(a_2h_2)\oplus\cO_F(b_1h_1+b_2h_2)
$$
where $0\le a_2\le 2$, $\vert b_1-b_2\vert\le 2$ and at least one of the $b_i$'s is not positive.

Twisting Sequence \eqref{seqRestriction} by $\cO_\Phi(-a_2\eta_2)$ we obtain sequence
\begin{equation}
\label{seqRestrictionTwisted}
0\longrightarrow \cG(-\eta-a_2\eta_2)\longrightarrow \cG(-a_2\eta_2)\longrightarrow \cO_F\oplus\cO_F(b_1h_1+(b_2-a_2)h_2)\longrightarrow 0
\end{equation}

We first assume $h^1\big(F,\cO_F((t+b_1)h_1+(t+b_2-a_2)h_2)\big)=0$ for each $t\le-1$. Hence the cohomology of the above sequence twisted by $\cO_\Phi(t\eta)$ implies the existence of surjective maps
$$
H^1\big(\Phi,\cG(t\eta-a_2\eta_2)\big)\twoheadrightarrow H^1\big(\Phi,\cG((t+1)\eta-a_2\eta_2)\big)
$$
for each $t\le-1$. The vanishing $h^1\big(\Phi,\cG(t\eta-a_2\eta_2)\big)=0$ for $t\ll-1$, thus forces $h^1\big(\Phi,\cG(t\eta-a_2\eta_2)\big)=0$ for each $t\le-1$. It follows that $h^0\big(\Phi,\cG(-a_2\eta_2)\big)=1+h^0\big(F,\cO_F(b_1h_1+(b_2-a_2)h_2)\big)\ne0$.

If $\sigma\in H^0\big(\Phi,\cG(-a_2\eta_2)\big)$ is a general section then we can write $(\sigma)_0=\Theta\cup\Xi\subseteq\Phi$, where $\Theta$ has pure codimension $2$ (or it is empty) and $\Xi$ is an effective divisor (or it is empty).

Consider the restriction morphism
$$
\psi\colon H^0\big(\Phi,\cG(-a_2\eta_2)\big)\to H^0\big(F,\cO_F\oplus\cO_F(b_1h_1+(b_2-a_2)h_2)\big).
$$
The cohomology of Sequence \eqref{seqRestrictionTwisted} and the vanishing $h^1\big(\Phi,\cG(-\eta-a_2\eta_2)\big)=0$ imply that $\psi$ is an isomorphism because we are assuming that $\cG$ is initialized. Thus $s:=\psi(\sigma)\in H^0\big(F,\cO_F\oplus\cO_F(b_1h_1+(b_2-a_2)h_2)\big)$ is general. We conclude that $(\sigma)_0\cap F=(s)_0=\emptyset$. Thus $(\sigma)_0=\emptyset$ too. In particular Sequence \eqref{seqIdeal} for $\sigma$ is
$$
0\longrightarrow\cO_\Phi\longrightarrow\cG(-a_2\eta_2)\longrightarrow\cO_\Phi(c_1(\cG)-2a_2\eta_2)\longrightarrow0.
$$
Since it corresponds to an element of
$$
\Ext^1_\Phi\big(\cO_\Phi(c_1(\cG)-2a_2\eta_2),\cO_\Phi\big)\cong H^1\big(\Phi,\cO_\Phi(2a_2\eta_2-c_1(\cG))
\big).
$$
The last space is zero, due to K\"unneth formula, thus the sequence above splits, whence $\cG\cong \cO_\Phi(a_2\eta_2)\oplus\cO_\Phi(c_1(\cG)-a_2\eta_2)$.

Now assume that there is at least one $t\le-1$ such that $h^1\big(F,\cO_F((t+b_1)h_1+(t+b_2-a_2)h_2)\big)\ne0$. Only one case is actually admissible as the following lemma shows.

\begin{lemma}
\label{lVanishing}
Let $a_2,b_1,b_2\in\bZ$ be such that $0\le a_2\le 2$, $\vert b_1-b_2\vert\le 2$ and at least one of the $b_i$'s is not positive. There is $t\le -1$ such that  $h^1\big(F,\cO_F((t+b_1)h_1+(t+b_2-a_2)h_2)\big)\ne0$, if and only if $a_2=1$, $b_1=2$, $b_2=0$ and $t=-1$.
\end{lemma}
\begin{proof}
Notice that $b_1,b_2\le2$ thanks to the hypothesis.
One implication is trivial. Indeed if $a_2=1$, $b_1=2$, $b_2=0$ and $t=-1$, then $h^1\big(F,\cO_F((t+b_1)h_1+(t+b_2-a_2)h_2)\big)=1$ thanks to Proposition \ref{pLineBundle}. 

Let us now prove the opposite implication. Assume that $h^1\big(F,\cO_F((t+b_1)h_1+(t+b_2-a_2)h_2)\big)\ne0$ for some $t\le-1$. 

Let $b_1\le b_2-a_2$. by Proposition \ref{pLineBundle} above we have
$$
t+b_1\le-2,\qquad 2t+b_1+b_2-a_2+1\ge0.
$$
Hence
$$
0\le 2t+b_1+b_2-a_2+1=(t+b_1)+t+b_2-a_2+1\le b_2-a_2-2\le -a_2\le0,
$$
for $t\le-1$. Indeed we know that $\vert b_1-b_2\vert\le2$, thus $b_2-2\le b_1$. If $b_1\le0$, then $b_2-2\le0$. If $b_1>0$, then $b_2\le0$ thanks to the hypothesis. We conclude that all the inequalities above must be actually equalities, whence $t=-1$, $t+b_1=-2$, $b_2=2$, $a_2=0$, thus $b_1=-1$. Since $3=\vert b_1-b_2\vert> 2$, we conclude that such a case cannot occur.

Let $b_1> b_2-a_2$. As in the previous case we deduce that
\begin{gather*}
t+b_2-a_2\le-2,\qquad 2t+b_1+b_2-a_2+1\ge0,\\
0\le 2t+b_1+b_2-a_2+1=(t+b_2-a_2)+t+b_1+1\le b_1-2\le0
\end{gather*}
whence $t=-1$, $b_1=2$, $t+b_2-a_2=-2$. Taking into account of the restrictions on $a_2$ and $b_2$ we have either $b_2=-1$ and $a_2=0$, or $b_2=0$ and $a_2=1$. Again the case $b_2=-1$ and $a_2=0$ cannot occur because $3=\vert b_1-b_2\vert> 2$.
\end{proof}

From the above lemma we may assume that $a_2=1$, $b_1=2$, $b_2=0$ i.e.
$\cE=\cO_F(h_2)\oplus\cO_F(2h_1)$. Notice that $\cG^\vee(\eta)$ is indecomposable: it is also aCM by Serre's duality, because $\omega_\Phi\cong\cO_\Phi(-3\eta)$. Moreover, in this case
$$
\cG^\vee(\eta)\otimes\cO_F\cong\cE^\vee(h)\cong \cO_F(h_1)\oplus\cO_F(-h_1+h_2)
$$
is initialized, thus $\cG^\vee(\eta)$ is also initialized. Now we can substitute $\cG$ with $\cG^\vee(\eta)$: since $h^1\big(F, \cO_F(th-2h_1+h_2)\big)=0$ for each $t\le -1$, it follows that we can repeat almost verbatim (we only have to permute the roles of $\eta_1$ and $\eta_2$) the first part of the above discussion, proving that $\cG^\vee(\eta)$, hence $\cG$, must be decomposable.

We summarize the above discussion in the following statement.

\begin{theorem}
\label{tVectorBundle}
If $\cG$ is an initialized, indecomposable, aCM bundle of rank $2$ on $\Phi$, then $\cG\otimes\cO_F$ is an initialized, indecomposable, aCM bundle of rank $2$ on $F$.
\end{theorem}

Taking into account the above result, in what follows we will first deal with initialized, indecomposable, aCM bundles of rank $2$ on $F$, giving their complete classification. Then we will use the results obtained on $F$ in order to classify also initialized, indecomposable, aCM bundles of rank $2$ on $\Phi$ and viceversa.

\section{Theorem A: the lower bound}
\label{sLowerChern6}
In this section we will find a bound from below for the first Chern
class $c_1:=\alpha_1h_1+\alpha_2h_2$ of an indecomposable, initialized, aCM bundle $\cE$ of rank
$2$ on $F$. Due to Theorem \ref{tVectorBundle}, such a bound also holds for the first Chern class of each initialized indecomposable aCM bundle of rank $2$ on $\Phi$.

\begin{lemma}
  \label{lGG6}
  Let $\cE$ be an initialized aCM bundle of rank $2$ on $F$. Then $\cE^\vee(2h)$ is aCM and globally generated.
\end{lemma}
\begin{proof}
Recall that $\omega_F\cong\cO_F(-2h)$, thus Serre's duality implies that $\cE^\vee$ is aCM. Moreover, we have $h^i\big(F,\cE^\vee((2-i)h)\big)=h^{3-i}\big(F,\cE((i-4)h)\big)=0$, $i=1,2,3$, thus $\cE$ is $2$--regular in the sense of Castelnuovo--Mumford (see \cite{Mu}), hence the second assertion follows.
\end{proof}

It follows that $4h-c_1=c_1(\cE^\vee(2h))$ is effective on $F$, whence $\alpha_i\le4$. Moreover if $\alpha_1,\alpha_2\ge3$, then there would exists an injective morphism $H^0\big(F,\cE(2h-c_1)\big)\to H^0\big(F,\cE(-h)\big)$. On the one hand we know that the target space is zero. On the other hand $H^0\big(F,\cE(2h-c_1)\big)\ne0$ since $\cE(2h-c_1)\cong{\cE}^\vee(2h)$ is globally generated, a contradiction. Thus we can always assume $\alpha_1\le2$.

We first check that the zero--locus of each section of an indecomposable, initialized, aCM bundle of rank $2$ on $F$ is non--empty.

\begin{lemma}
\label{lNonEmpty6}
Let $\cE$ be an indecomposable, initialized, aCM bundle of rank $2$ on $F$. Then the zero locus $(s)_0$ of a section of $\cE$ is non--empty.
\end{lemma}
\begin{proof}
We assume  $\alpha_1\le\alpha_2$. If $(s)_0=\emptyset$, then sequence \eqref{seqIdeal} becomes
$$
0\longrightarrow \cO_F\longrightarrow \cE\longrightarrow \cO_F(c_1)\longrightarrow 0.
$$
The class of such an extension should be a non--zero element in $H^1\big(F,\cO_F(-c_1)\big)$. Thanks to Proposition \ref{pLineBundle}, we deduce that $\alpha_2\ge2$ and $\alpha_1+\alpha_2\le1$, whence $\alpha_1\le-1$. The cohomology of the above sequence twisted by $\cO_F((-2-\alpha_1) h)$ and Proposition \ref{pLineBundle} yields
$$
h^1\big(F,\cE((-2-\alpha_1)h)\big)=h^1\big(F,\cO(-2h_1+(-2+\alpha_2-\alpha_1)h_2)\big)\ge1.
$$
We conclude that $\cE$ would not be aCM, a contradiction.
\end{proof}

We know that for each non--zero $s\in H^0\big(F,\cE\big)$ we have $(s)_0=E\cup D$ where $E$ has pure codimension $2$ (or it is empty) and $D\in\vert\delta_1h_1+\delta_2h_2\vert$ has pure codimension $1$ (or it is empty). Moreover at least one of them is non--empty thanks to Lemma \ref{lNonEmpty6}. Sequence \eqref{seqIdeal} for $s$ thus becomes
$$
  0\longrightarrow \cO_F(D)\longrightarrow \cE\longrightarrow \cI_{E\vert F}(c_1-D)\longrightarrow 0.
$$

On the one hand, twisting Sequence \eqref{seqIdeal} for $s$ by $\cO_F(-h)$ and
taking its cohomology, the vanishing of $h^0\big(F,\cE(-h)\big)$
implies $h^0\big(F,\cO_F(D-h)\big)=0$. In particular we know that at
least one of the $\delta_i$ is zero.

On the other hand, twisting the same sequence by $\cO_F(2h-c_1)$, using the isomorphism $\cE(-c_1)\cong{\cE}^\vee$ and taking into account that ${\cE}^\vee(2h)$ is globally generated (see Lemma \ref{lGG6}), we obtain that $\cI_{E\vert F}(2h-D)$ is globally generated too. Thus
$$
0\ne H^0\big(F, \cI_{E\vert F}(2h-D)\big)\subseteq H^0\big(F, \cO_F(2h-D)\big),
$$
hence $\delta_i\le2$, $i=1,2$. We conclude that $(\delta_1,\delta_2)\in\left\{\ (0,0), (0,1), (0,2), (1,0), (2,0)\ \right\}$: in particular $\cO_F(D)$ is aCM (see Corollary \ref{cLineBundle}).

We will now prove that $c_1-D$ is effective: hence the same holds for $c_1$. We start with the following lemma.

\begin{lemma}
  Let $\cE$ be an indecomposable, initialized, aCM bundle of rank $2$ on $F$. Assume that $s\in H^0\big(F,\cE\big)$ is such that $(s)_0=E\cup D$ where $E$ has codimension $2$ (or it is empty) and $D$ has codimension $1$. If $c_1-D$ is not effective, then $E=\emptyset$.
  \end{lemma}
  \begin{proof}
Assume $c_1-D$ is non--effective: thanks to Sequence \eqref{seqStandard} for $s$,
\begin{equation}
\label{NUM0}
h^0\big(F,\cI_{E\vert F}(c_1-D-th)\big)\le h^0\big(F,\cO_{F}(c_1-D-th)\big)=0
\end{equation}
for each positive integer $t$.  Taking into account that $\cE$ and $\cO_F(D)$ are aCM and $\delta_i\ge0$, the cohomology of Sequence \eqref{seqIdeal} for $s$ yields also the vanishings
\begin{equation}
\label{NUM1}
\begin{gathered}
  h^1(F,\cI_{E\vert F}(c_1-D-th))=0,\qquad t\in\ \bZ\\
  h^2(F,\cI_{E\vert F}(c_1-D))=h^3\big(F,\cO_F(D)\big)=0\\
  h^2(F,\cI_{E\vert F}(c_1-D-h))=h^3\big(F,\cO_F(D-h)\big)=0.
  \end{gathered}
\end{equation}
Finally, we also have
\begin{equation}
\label{NUM2}
h^2\big(F,\cI_{E\vert F}(c_1-D-2h)\big)\le h^3\big(F,\cO_F(D-2h)\big)=h^0\big(F,\cO_F(-D)\big)\le1.
\end{equation}

 Assume $E\ne\emptyset$, whence $\deg(E)\ge1$. Let $H$ be a general hyperplane in $\p7$. Define $S:=F\cap H$ and $Z:=E\cap H$, so that $\dim(Z)=0$. We have the following exact sequence
  \begin{equation*}
    0\longrightarrow \cI_{E\vert F}(c_1-D-h)\longrightarrow \cI_{E\vert F}(c_1-D)\longrightarrow \cI_{Z\vert S}(c_1-D)\longrightarrow 0.
  \end{equation*}
  The cohomology of the above sequence, Equalities \eqref{NUM0}, \eqref{NUM1} and Inequality \eqref{NUM2} imply
\begin{gather*}
h^0\big(S,\cI_{Z\vert S}(c_1-D)\big)=h^1\big(S,\cI_{Z\vert S}(c_1-D)\big)=0,\\
h^0\big(S,\cI_{Z\vert S}(c_1-D-h)\big)=0,\qquad
h^1\big(S,\cI_{Z\vert S}(c_1-D-h)\big)\le1.
\end{gather*}

  The above relations and the cohomology of Sequence 
  \begin{equation}
  \label{seqSectionScheme}
  0\longrightarrow \cI_{Z\vert S}(c_1-D)\longrightarrow \cO_S(c_1-D)\longrightarrow \cO_Z\longrightarrow 0
 \end{equation}
give
\begin{equation}
\label{NUM3}
h^1\big(S,\cO_{S}(c_1-D)\big)=0,
\end{equation}
\begin{equation}
\label{NUM6}
h^0\big(S,\cO_{S}(c_1-D)\big)=h^0\big(Z,\cO_Z\big)=\deg(Z)=\deg(E).
\end{equation}
The same relations as above and the cohomology of Sequence \eqref{seqSectionScheme} twisted by $\cO_S(-h)$ also yield
$$
h^0\big(Z,\cO_Z\big)\le h^0\big(S,\cO_{S}(c_1-D-h)\big)-h^1\big(S,\cO_{S}(c_1-D-h)\big)+1.
$$

Looking at the formulas for $h^0\big(Z,\cO_Z\big)$ obtained above, we deduce the inequality
\begin{equation}
\label{NUM5}
1\ge h^0\big(S,\cO_{S}(c_1-D)\big)-h^0\big(S,\cO_{S}(c_1-D-h)\big)+h^1\big(S,\cO_{S}(c_1-D-h)\big).
\end{equation}
The remaining part of the proof is devoted to show that the aforementioned inequality leads to a contradiction.

Set $c_1-D:=\epsilon_1h_1+\epsilon_2h_2$ and assume $\epsilon_1\le\epsilon_2$. Obviously $\epsilon_1\le -1$ because $c_1-D$ is assumed to be not effective. Thus the cohomology of sequence
\begin{equation}
\label{seqSectionSurface}
0\longrightarrow\cO_F(c_1-D-h)\longrightarrow\cO_F(c_1-D)\longrightarrow\cO_S(c_1-D)\longrightarrow0
\end{equation}
and Equality \eqref{NUM6} yield
$$
1\le \deg(E)=h^0\big(S,\cO_{S}(c_1-D)\big)\le h^1\big(F,\cO_{F}(c_1-D-h)\big).
$$
Hence Proposition \ref{pLineBundle} implies
$$
h^1\big(F,\cO_{F}(c_1-D-h)\big)=-\frac{\epsilon_1\epsilon_2(\epsilon_1+\epsilon_2)}{2},
$$
$\epsilon_1\le-1$ and $\epsilon_1+\epsilon_2\ge1$, whence $\epsilon_2\ge 2$. Since $\cE$, hence $\cE^\vee$, is aCM and $h^0(F,\cI_{E\vert F}(-D))=0$, it follows from the cohomology of sequence
$$
0\longrightarrow \cO_F(D-c_1)\longrightarrow \cE^\vee\longrightarrow \cI_{E\vert F}(-D)\longrightarrow 0
$$
that $h^1\big(F,\cO_F(-\epsilon_1h_1-\epsilon_2h_2)\big)=h^1\big(F,\cO_F(D-c_1)\big)=0$. Taking into account that $-\epsilon_2\le-\epsilon_1$ and $\epsilon_2\ge 2$ whence $-\epsilon_2\le -2$, again by Proposition \ref{pLineBundle} we obtain
\begin{equation}
\label{LowerBound}
\epsilon_1+\epsilon_2\ge2.
\end{equation}
Recall that $\epsilon_1-2\le-3$: if $\epsilon_1+\epsilon_2\ge3$, then $(\epsilon_1-2)+(\epsilon_2-2)+1\ge0$, hence again Proposition \ref{pLineBundle} yields
$$
h^1\big(F,\cO_{F}(c_1-D-2h)\big)=-\frac{(\epsilon_1-1)(\epsilon_2-1)(\epsilon_1+\epsilon_2-2)}{2}.
$$
The same equality still holds when $\epsilon_1+\epsilon_2=2$, because both the members of the equality above are zero. Finally, Proposition \ref{pLineBundle} yields
$$
h^1\big(F,\cO_{F}(c_1-D)\big)=-\frac{(\epsilon_1+1)(\epsilon_2+1)(\epsilon_1+\epsilon_2+2)}{2}
$$
if $\epsilon_1\le-2$. Again the same equality holds also for $\epsilon_1=-1$, due to the vanishing of both the members.

Notice that if $\epsilon_1\le-1$ and $\epsilon_1+\epsilon_2\ge2$, then $h^2\big(F,\cO_{F}(c_1-D-2h)\big)=0$ (Proposition \ref{pLineBundle}). Thus Equality \eqref{NUM3} and  the cohomology of sequence \eqref{seqSectionSurface} yield
\begin{gather*}
h^0\big(S,\cO_{S}(c_1-D)\big)=h^1\big(F,\cO_{F}(c_1-D-h)\big)-h^1\big(F,\cO_{F}(c_1-D)\big),\\
{\begin{align*}
h^0\big(S,\cO_{S}(c_1-D-h)\big)&-h^1\big(S,\cO_{S}(c_1-D-h)\big)=\\
&=h^1\big(F,\cO_{F}(c_1-D-2h)\big)-h^1\big(F,\cO_{F}(c_1-D-h)\big).
\end{align*}}
\end{gather*}
In particular, by substituting the above expressions in Inequality \eqref{NUM5}, we obtain
$$
1\ge 2h^1\big(F,\cO_{F}(c_1-D-h)\big)-h^1\big(F,\cO_{F}(c_1-D)\big)-h^1\big(F,\cO_{F}(c_1-D-2h)\big).
$$
Thus the expressions obtained above in terms of the $\epsilon_i$'s for $h^1\big(F,\cO_{F}(c_1-D-th)\big)$, $t=0,1,2$, and  the Inequality \eqref{LowerBound} finally yield $1\ge 3(\epsilon_1+\epsilon_2)\ge6$, a contradiction. We conclude that if $c_1-D$ is non--effective, then $E=\emptyset$.
\end{proof}

We conclude the section by  proving the claimed effectiveness of $c_1-D$.

\begin{proposition}
\label{pLower6}
If $\cE$ is an indecomposable, initialized, aCM bundle of rank $2$ on $F$, $s\in H^0\big(F,\cE\big)$ and $D$ denotes the component of pure codimension $1$ of $(s)_0$, if any, then $c_1\ge c_1-D\ge0$.

In particular, if $c_1=0$, then the zero--locus of each section $s\in H^0\big(F,\cE\big)$ has pure codimension $2$.
\end{proposition}
\begin{proof}
Let $c_1-D:=\epsilon_1h_1+\epsilon_2h_2$ be non--effective: we can assume $\epsilon_1\le-1$. The previous Lemma implies $E=\emptyset$ whence $\cI_{E\vert F}\cong\cO_F$. Sequence \eqref{seqIdeal} thus becomes
$$
0\longrightarrow \cO_F(D)\longrightarrow \cE\longrightarrow \cO_{F}(c_1-D) \longrightarrow 0.
$$
Recall that  $(\delta_1,\delta_2)\in\left\{\ (0,0), (0,1), (0,2), (1,0), (2,0)\ \right\}$. The cohomology of the above sequence twisted by $\cO_F(th)$ and Corollary \ref{cLineBundle} imply
$$
0=h^1\big(F,\cE(th)\big)=h^1\big(F,\cO_{F}(c_1-D+th)\big),\qquad t\in\bZ.
$$
Moreover the above sequence does not split because $\cE$ is assumed to be indecomposable, hence
$$
H^1(F,\cO_{F}(2D-c_1))\cong Ext^1_F\big(\cO_{F}(c_1-D), \cO_{F}(D))\big)\ne0.
$$
We have $2D-c_1\in\left\vert (\delta_1-\epsilon_1)h_1+(\delta_2-\epsilon_2)h_2\right\vert$. Since $\delta_i\ge0$, $i=1,2$, and $\epsilon_1\le -1$, it follows that $\delta_1-\epsilon_1\ge1$. Proposition \ref{pLineBundle} thus gives $\delta_2-\epsilon_2\le-2$, whence $\epsilon_2\ge2$. We conclude that $\epsilon_2-\epsilon_1\ge3$.

If we take $t:=-2-\epsilon_1$, then $c_1-D+th=-2h_1+(\epsilon_2-\epsilon_1-2)h_2$, hence again Proposition \ref{pLineBundle} implies $h^1\big(F,\cO_{F}(c_1-D+th)\big)\ne0$, a contradiction.
\end{proof}

\section{Theorem A: the upper bound and zero--loci of general sections}
\label{sUpperChern6}
In this section we will find a bound from above for the first Chern class $c_1=\alpha_1h_1+\alpha_2h_2$ of an indecomposable, initialized, aCM bundle $\cE$ of rank $2$ on $F$. To this purpose, from now on we will assume that the second Chern class of $\cE$ is $c_2:=\beta_1h_2^2+\beta_2h_1^2$.

If $\omega_2$ is the second Chern class of the sheaf $\Omega_{F}^1$, Riemann--Roch theorem yields
\begin{equation}
  \label{RRgeneral}
    \chi(\cE)=2+{1\over6}(c_1^3-3c_1c_2)-{1\over2}(c_1^2h-2c_2h)+{1\over{12}}(4c_1h^2+\omega_2c_1).
\end{equation}

Easy computations yield
\begin{gather*}
\label{chern}
c_1^3=3(\alpha_1^2\alpha_2+\alpha_1\alpha_2^2),\\
c_1^2h=\alpha_1^2+4\alpha_1\alpha_2+\alpha_2^2,\\
c_1h^2=3(\alpha_1+\alpha_2).
\end{gather*}
Recall that $p_1$ is isomorphic to the canonical map $\bP(\Omega_{\p2}^1(2))\to \p2$, hence it is smooth. In particular we have the  two exact sequences
\begin{gather*}
0\longrightarrow\Omega^1_{F\vert\bP^2}\longrightarrow p_1^*\Omega^1_{\bP^2}(2)\otimes \cO_F(-h_2)\longrightarrow\cO_F\longrightarrow0,\\
0\longrightarrow p_1^*\Omega^1_{\bP^2}\longrightarrow\Omega^1_{F}\longrightarrow\Omega^1_{F\vert\bP^2}\longrightarrow0
\end{gather*}
Another Chern class computation thus gives $\omega_2=6h_1h_2$, hence $\omega_2c_1=6(\alpha_1+\alpha_2)$.

We have $h^i\big(F,\cE^\vee(-h)\big)=h^{3-i}\big(F,\cE(-h)\big)=0$ for $i\ge1$. Since
$$
h^0\big(F,\cI_{E\vert F}(-D-h)\big)\le h^0\big(F,\cO_{F}(-D-h)\big)=0
$$
(see Sequence \eqref{seqStandard}), it follows that
$$
\chi\big(\cE^\vee(-h)\big)=h^0\big(F,\cE^\vee(-h)\big)=h^0\big(F,\cO_F(D-c_1-h)\big)=0
$$
(due to sequence \eqref{seqIdeal} and to the effectiveness of $c_1-D$ proved in Proposition \ref{pLower6}). Such an equality, Formula \eqref{RRgeneral} for $\cE^\vee(-h)$ and Equalities \eqref{chern} yield
\begin{equation}
\label{c_1c_2}
c_1c_2=\alpha_1\beta_1+\alpha_2\beta_2=\alpha_1^2\alpha_2+\alpha_1\alpha_2^2.
\end{equation}

We also have $h^i\big(F,\cE^\vee\big)=h^{3-i}\big(F,\cE(-2h)\big)=0$, for $i\ge 1$. If $D=0$, then $E\ne\emptyset$, hence $h^0\big(F,\cI_{E\vert F}\big)=h^0\big(F,\cI_{E\vert F}(-D)\big)=0$. If $D\ne0$, then
$$
h^0\big(F,\cI_{E\vert F}(-D)\big)\le h^0\big(F,\cO_{F}(-D)\big)=0.
$$
Let
$$
e(c_1,D):=\left\lbrace\begin{array}{ll} 
    0\quad&\text{if $D\ne c_1$,}\\
    1\quad&\text{if $D=c_1$.}
  \end{array}\right.
$$
Since $(s)_0\ne\emptyset$, it follows that 
$$
\chi(\cE^\vee)=h^0\big(F,\cE^\vee\big)=h^0\big(F,\cO_F(D-c_1)\big)=e(c_1,D)
$$
Combining as above Sequences \eqref{seqStandard}, \eqref{seqIdeal} and Formulas \eqref{RRgeneral} for $\cE^\vee$ and \eqref{c_1c_2}, we finally obtain
\begin{equation}
\label{hc_2}
hc_2=\beta_1+\beta_2=2+\frac{1}{2}\left(\alpha_1^2+4\alpha_1\alpha_2+\alpha_2^2-3\alpha_1-3\alpha_2\right)-e(c_1,D).
\end{equation}

\begin{lemma}
  \label{lInvariantE}
  Let $\cE$ be a vector bundle of rank $2$ on $F$ and $s\in H^0\big(F,\cE\big)$ a section such that $E:=(s)_0$ has pure codimension $2$. Then
  $$
  \deg(E)=h c_2,\qquad
  p_a(E)=\frac12c_1c_2-hc_2+1.
  $$
 \end{lemma}
 \begin{proof}
 The result was proved in \cite{C--F--M1} (see Lemma 2.1) when $F\cong\p1\times\p1\times\p1$. The same proof holds verbatim also in this case, because it only depends on the isomorphism $\omega_F\cong\cO_F(-2h)$.
 \end{proof}
 
Assume that $E\ne\emptyset$. Its class in $A^2(F)$ is $c_2(\cE(-D))=c_2-c_1D+D^2$. Since $\vert h_i\vert$, $i=1,2$ is base--point--free on $F$, $D\ge0$ and $c_1-D$ is effective (see Proposition \ref{pLower6}), it follows that
\begin{equation}
  \label{positivity3}
  \beta_i\ge \beta_i-h_i(c_1-D)D=h_i c_2-h_i(c_1-D)D=h_i c_2(\cE(-D))\ge0,\qquad i=1,2.
\end{equation}

\begin{proposition}
  \label{pUpper6}
  If $\cE$ is an indecomposable, initialized, aCM bundle of rank $2$ on $F$, then $2h-c_1\ge0$.

  Moreover, if equality holds, then $\cE$ is Ulrich, hence the zero--locus $E:=(s)_0$ of a general section $s\in H^0\big(F,\cE\big)$ is a smooth elliptic curve of degree $8$.
\end{proposition}
\begin{proof}
We distinguish two cases according to whether $\cE$ is regular in the sense of Castelnuovo--Mumford or not.

In the latter case we infer $h^0\big(F,\cE^\vee(h)\big)=h^3\big(F,\cE(-3h)\big)\ne0$ because $\cE$ is aCM. Thus, if $t\in\bZ$ is such that $\cE^\vee(th)$ is initialized, we know that $t\le1$. Since $\cE^\vee(th)$ is aCM too, we know that  $2th-c_1=c_1(\cE^\vee(th))$ is effective for some $t\le1$, due to Proposition \ref{pLower6}. We conclude that $2h-c_1$ is effective too.

Let $\cE$ be regular, whence globally generated. Thus the zero--locus $E:=(s)_0$ of a general section $s\in H^0\big(F,\cE\big)$ is a smooth curve: in particular its divisorial part $D$ is zero. Let $\alpha_1\le\alpha_2$: we already know that $0\le \alpha_1\le2$ and $\alpha_2\le 4$. We are interested in the cases $\alpha_2=3,4$, so that $c_1\ne0$.

We have $h^i\big(F,\cE^\vee(h)\big)=h^{3-i}\big(F,\cE(-3h)\big)=0$, $i=0,1,2,3$, because $\cE$ is aCM, regular and initialized. Combining the above remarks with equality \eqref{RRgeneral}, Formulas  \eqref{c_1c_2}, \eqref{hc_2} and the vanishing of $e(c_1,D)$ (recall that $c_1\ne0$ and $D=0$), we finally obtain
$$
0=\chi(\cE^\vee(h))=12-3\alpha_1-3\alpha_2.
$$
Thus $(\alpha_1,\alpha_2)\in\{\ (1,3), (0,4)\ \}$.

If $c_1=h_1+3h_2$, then Formulas \eqref{hc_2} and \eqref{c_1c_2} give
$$
\beta_1+\beta_2=  h c_2=7,\qquad c_1c_2=\beta_1+3\beta_2=12.
$$
subtracting the two equation each other we obtain $2\beta_2=5$, a contradiction: thus such a case cannot occur.

We will now prove that if  $c_1$ is $4h_2$, then $\cE$ splits as sum of invertible sheaves, thus such a case cannot occur too. Equalities \eqref{c_1c_2}, \eqref{hc_2} and the Inequalities \eqref{positivity3} force $c_2=4h_2^2$. Lemma \eqref{lInvariantE} implies $\deg(E)=4$ and $p_a(E)=-3$. 

Recall that for each pair of curves $C'$ and $C''$ without common components, the arithmetic genera of $C'$, $C''$ and $C'\cup C''$ satisfy the equality
$$
p_a(C'\cup C'')=p_a(C')+p_a(C'')+c-1
$$
where $c$ is the degree of the intersection $C'\cap C''$. It follows that $E$, being smooth, is necessarily the union of $4$ pairwise skew lines whose cycle in $A^2(F)$ is $h_2^2$ (indeed it is easy to check that the class of each line on $F$ is either $h_1^2$, or $h_2^2$).

Consider the projection $p_2\colon F\cong\bP(\Omega_{\p2}^1(2))\to\p2$ and let $\Gamma:=p_2(E)\subseteq\p2$. The scheme $\Gamma$ is a set of $4$ pairwise distinct points and $ p_2^*\cI_{\Gamma\vert\p2}\cong\cI_{E\vert F}$. Since $F\cong\bP(\Omega_{\p2}^1(2))$, it follows that $ p_2{}_*\cO_F(h_1)\cong \Omega_{\p2}^1(2)$. Moreover $\cO_F(h_2)=p_2^*\cO_{\p2}(1)$, hence
\begin{align*}
h^0\big(\p2,\cI_{\Gamma\vert\p2}(2)\otimes \Omega_{\p2}^1(1)\big)&=h^0\big(\p2,\cI_{\Gamma\vert\p2}(1)\otimes \Omega_{\p2}^1(2)\big)=\\
&=h^0\big(\p2,\cI_{\Gamma\vert\p2}\otimes \Omega_{\p2}^1(2)\otimes  p_2{}_*\cO_F(h_2)\big)=\\
&=h^0\big(F, p_2^*\cI_{\Gamma\vert\p2}\otimes \cO_F(h_1)\otimes \cO_F(h_2)\big)=h^0\big(F,\cI_{E\vert F}(h)\big).
\end{align*}
In order to compute such a last dimension we use the cohomology of Sequence \eqref{seqIdeal} twisted by $\cO_F(h-4h_2)$. Thanks to Proposition \ref{pLineBundle} we have $h^0\big(\p2,\cI_{E\vert F}(h)\big)=h^0\big(\p2,\cE^\vee(h)\big)=h^3\big(\p2,\cE(-3h)\big)$, which is zero because we are assuming $\cE$ regular.

The points of $\Gamma$ are either in general position, or three of them lie on a line not containing the fourth one, or they are aligned.

In the first case, we have the Koszul resolution
\begin{equation}
\label{Koszul}
0\longrightarrow\cO_{\p2}(-2)\longrightarrow\cO_{\p2}^{\oplus2}\longrightarrow\cI_{\Gamma\vert\p2}(2)\longrightarrow0,
\end{equation}
We have both Sequence \eqref{seqIdeal} and the pull back of the Koszul resolution above via $p_2^*$, i.e.
  \begin{gather*}
    0\longrightarrow\cO_F\longrightarrow \cE\longrightarrow  \cI_{E\vert F}(4h_2)\longrightarrow0,\\
    0\longrightarrow\cO_F\longrightarrow \cO_F(2h_2)^{\oplus2}\longrightarrow  \cI_{E\vert F}(4h_2)\longrightarrow0,
  \end{gather*}
(we used the isomorphisms $ p_2^*\cI_{\Gamma\vert\p2}\cong\cI_{E\vert F}$ and $ p_2^*\cO_{\p2}(1)\cong\cO_{F}(h_2)$ in order to obtain the second sequence). We conclude by observing that Theorem \ref{tSerre} yields $\cE\cong \cO_F(2h_2)^{\oplus2}$, because $h^1\big(F,\cO_F(-4h_2)\big)=0$ (again by Proposition \ref{pLineBundle}).

In the second case, up to a proper linear change of coordinates, we can assume that the homogeneous ideal of $\Gamma$ is $(x_0x_1,x_0x_2,x_1x_2(x_1-x_2))$: it follows the existence of a resolution of the form
$$
0\longrightarrow\cO_{\p2}(-1)\oplus\cO_{\p2}(-2)\longrightarrow\cO_{\p2}^{\oplus2}\oplus\cO_{\p2}(-1)\longrightarrow\cI_{\Gamma\vert\p2}(2)\longrightarrow0,
$$
Consider Beilinson Theorem in its second form (see \cite{Hu}, Proposition 8.28). Using the above sequence, Bott formulas and the vanishing $h^0\big(\p2,\cI_{\Gamma\vert\p2}(2)\otimes \Omega_{\p2}^1(1)\big)=0$, we see that the only non zero terms are
\begin{gather*}
E_1^{0,0}=H^0\big(\p2,\cI_{\Gamma\vert\p2}(2)\otimes \cO_{\p2}\big)\otimes\cO_{\p2}\cong\cO_{\p2}^{\oplus2},\\
E_1^{-2,1}=H^1\big(\p2,\cI_{\Gamma\vert\p2}(2)\otimes \cO_{\p2}(-1)\big)\otimes\cO_{\p2}(-2)\cong\cO_{\p2}(-2).
\end{gather*}
Since Beilinson spectral sequence abuts to
$$
E^{r,s}_{\infty}\cong\left\lbrace\begin{array}{ll}
\cI_{\Gamma\vert \p2}(2)\qquad&\text{if $r+s=0$,}\\
0\qquad&\text{if $r+s\ne0$,}
 \end{array}\right.
$$
Thus we would have an exact sequence analogous to Sequence \eqref{Koszul}. In particular, $E$ would be the complete intersection of two conics, a contradiction.

Finally, in the third case we have the exact sequence
$$
0\longrightarrow\cO_{\p2}(-3)\longrightarrow\cO_{\p2}(1)\oplus\cO_{\p2}(-2)\longrightarrow\cI_{\Gamma\vert\p2}(2)\longrightarrow0.
$$
It is immediate to check that $h^0\big(\p2,\cI_{\Gamma\vert\p2}(1)\otimes \Omega_{\p2}^1(2)\big)=3$. Thus this case cannot occur.
\end{proof}

In order to complete the proof of Theorem A stated in the introduction, we have only to show that if $\cE$ is an indecomposable, initialized, aCM bundle, then its general section vanishes exactly along a curve. To this purpose we assume  $\alpha_1\le\alpha_2$: thanks to Propositions \ref{pLower6} and \ref{pUpper6} we can assume that $c_1\ne0,2h$, hence we restrict our attention to the cases $(0,1)$, $(0,2)$, $(1,1)$, $(1,2)$.

Let  $s\in H^0\big(F,\cE\big)$ be a general section and assume that satisfies $(s)_0=E\cup D$ where $E$ has codimension $2$ (or it is empty) and $D\in\vert\delta_1h_1+\delta_2h_2\vert$ is non--zero. We already know that  $  (\delta_1,\delta_2)\in\left\{\ (0,1), (0,2)\ \right\}$ up to permutations. For each value of $(\alpha_1,\alpha_2)$ and $(\delta_1,\delta_2)$ such that $c_1-D\ge0$ we can compute $hc_2$ and $c_1c_2$ with the help of Equalities \eqref{hc_2} and \eqref{c_1c_2}, hence also $(\beta_1,\beta_2)$. Taking into account that $c_2(\cE(-D))$ is the class of $E$ inside $A^2(F)$ we are finally able to write down the following table.
$$
\begin{array}{|c|c|c|c|c|c|c|}           \hline
  { (\alpha_1,\alpha_2) } & {(\delta_1,\delta_2) } & e(c_1-D) & hc_2 & c_1c_2 & {(\beta_1,\beta_2) } &E \\ \hline
  (0,1) & (0,1) & 1 & 0 & 0 & (0,0) & \emptyset \\  \hline
  (0,2) & (0,1) & 0 & 1 & 0 & (1,0) & \emptyset  \\  \hline
  (0,2) & (0,2) & 1 & 0 & 0 & (0,0) & \emptyset \\  \hline
  (1,1) & (0,1) & 0 & 2 & 2 & (1,1) & \emptyset \\  \hline
  (1,1) & (0,1) & 0 & 2 & 2 & (2,0) & h_2^2-h_1^2 \\  \hline
  (1,1) & (0,1) & 0 & 2 & 2 & (0,2) & -h_2^2+h_1^2 \\  \hline
  (1,2) & (0,1) & 0 & 4 & 6 & (2,2) & h_1^2 \\  \hline
  (1,2) & (0,2) & 0 & 4 & 6 & (2,2) & \emptyset \\  \hline
  (1,2) & (1,0) & 0 & 4 & 6 & (2,2) & \emptyset \\  \hline
\end{array}
$$

We know that $D$ and $c_1-D$ are effective on $F$, thus globally generated. Moreover Proposition \ref{pLineBundle} implies $h^1\big(F,\cO_F(D)\big)=0$. If $E=\emptyset$, then $\cI_{E\vert F}(c_1-D)\cong\cO_F(c_1-D)$.  thus  Sequence \eqref{seqIdeal} would imply that $\cE$ should be globally generated. But if this were the case, then the general section of $\cE$ would vanish along a curve, contradicting the initial hypothesis. 

In the sixth case, we would have $Eh_2=-1$, which is absurd, because the linear system $\vert h_2\vert$ is base--point--free. Thus the sixth case cannot occur. A similar argument shows that also the seventh case cannot occur as well.

We conclude that only the fifth case is possible.
We have the following exact Koszul resolution
$$
0\longrightarrow \cO_F\longrightarrow \cO_F(h_1)^{\oplus2}\longrightarrow\cI_{E\vert F}(2h_1)\longrightarrow0.
$$
It follows that $\cI_{E\vert F}(h)$ is globally generated. Sequence \eqref{seqIdeal} would imply that $\cE$ should be globally generated because $c_1-D=h$ and $h^1\big(F,\cO_F(h_2)\big)=0$. But we already checked that this leads to a contradiction, hence such a case cannot occur.

We summarize the above discussion in the following statement.

\begin{lemma}
\label{lCodimension}
  If $\cE$ is an indecomposable, initialized, aCM bundle of rank $2$ on $F$, then the zero--locus $(s)_0$ of a general section $s\in H^0\big(F,\cE\big)$ has pure codimension $2$.
\end{lemma}

We are now able to give the proofs of Theorems A stated in the introduction.

{\medbreak
\noindent
{\it Proofs of Theorems A for $F$ and $\Phi$.}
Propositions \ref{pLower6}, \ref{pUpper6} and Lemma \ref{lCodimension} prove Theorem A for $F$.

Now we turn our attention to an initialized, indecomposable, aCM bundle $\cG$ of rank $2$ on $\Phi$ with first Chern class $\gamma_1=\alpha_1\eta_1+\alpha_2\eta_2$. Thanks to Theorem \ref{tVectorBundle} the bundle $\cE:=\cG\otimes\cO_F$ is initialized, indecomposable and aCM of rank $2$: taking into account of Remark \ref{rRestriction} and that its first Chern class is $c_1=\alpha_1h_1+\alpha_2h_2$, also Theorem A for $\Phi$ is completely proved.
\qed}

\section{The extremal cases}
\label{sExtremal}
In this section we will analyze and characterize the rank $2$ bundles on $F$ and $\Phi$ whose first Chern class is extremal, i.e. either $0$, or $2h$, or $2\eta$. 

Let $\cG$ be an indecomposable, initialized, aCM vector bundle of rank $2$ on $\Phi$ and let us denote by $\gamma_1$ and $\gamma_2$ its Chern classes. If $\gamma_2=\mu_1\eta_2^2+\mu_2\eta_1^2+\mu_3\eta_1\eta_2$, then the restriction $\cE:=\cG\otimes\cO_F$ satisfies $c_2=(\mu_1+\mu_3)h_2^2+(\mu_2+\mu_3)h_1^2$. Let $\Sigma$ and $E$ be the zero loci of general sections of $\cG$ and $\cE$ respectively. Thanks to Lemma \ref{lCodimension}, $E$ has codimension $2$ inside $F$, hence the same is true for $\Sigma$, by Remark \ref{rRestriction}. In this case the class of  $\Sigma$  in $A^2(\Phi)$ is $\gamma_2$ so that
$$
\deg(\Sigma)=\gamma_2\eta^2=\mu_1+\mu_2+2\mu_3\ge0.
$$
Moreover $\vert \eta_i\vert$ is base--point--free for $i=1,2$, whence
\begin{equation}
\label{positivity4}
\mu_1=\gamma_2h_1^2\ge 0,\qquad \mu_2=\gamma_2h_2^2\ge 0,\qquad\mu_3=\gamma_2h_1h_2\ge0.
\end{equation}

Let $\gamma_1=0$. We know that the zero--locus  $\Sigma:=(\sigma)_0$ of a general section $\sigma\in H^0\big(\Phi,\cG\big)$ has codimension $2$ inside $\Phi$. Formula \eqref{hc_2} and Lemma \ref{lInvariantE} give that $E$, the hyperplane section of $\Sigma$, is a line, thus $\Sigma$ is a plane. Then looking at the second Chern class $c_2$ of $\cE$, either $\mu_1$ or $\mu_2$ is $1$, $\mu_3$ being zero. Hence the cycle of $\Sigma$ in $A^2(\Phi)$ is either $\eta_1^2$ or $\eta_2^2$.

Conversely, take a plane $\Sigma\subseteq\Phi$. Since $\deg(\Sigma)=1$ we can assume that its cycle in $A^2(\Phi)$ is $\eta_2^2$. We know that $\omega_{\Sigma}\cong\cO_{\Sigma}(-3\eta)$. Since $\omega_{\Phi}\cong\cO_{\Phi}(-3\eta)$, it follows that $\det(\cN_{\Sigma\vert \Phi})\cong\cO_{\Sigma}$ by adjunction formula on $\Phi$. Theorem \ref{tSerre} with $\mathcal L:=\cO_{\Phi}$ yields the existence of a unique vector bundle $\cG$ of rank $2$ fitting into a sequence of the form
$$
0\longrightarrow \cO_{\Phi}\longrightarrow \cG\longrightarrow \cI_{\Sigma\vert \Phi}\longrightarrow 0.
$$
Hence $h^1\big(\Phi,\cG(t\eta)\big)\le h^1\big(\Phi, \cI_{\Sigma\vert \Phi}(t\eta)\big)$, $t\in \bZ$. From the cohomology of the exact sequence
$$
0\longrightarrow \cI_{\Phi\vert\p8}\longrightarrow \cI_{\Sigma\vert\p8}\longrightarrow \cI_{\Sigma\vert \Phi}\longrightarrow 0
$$
and the vanishing $h^1\big(\Phi,\cI_{\Sigma\vert\p8}(t)\big)=h^2\big(\Phi,\cI_{\Phi\vert\p8}(t)\big)=0$ (recall that both $\Sigma$ and $\Phi$ are aCM), it follows that $h^1\big(\Phi, \cI_{\Sigma\vert \Phi}(t\eta)\big)=0$, i.e. $h^1\big(\Phi,\cG(t\eta)\big)=0$. A similar computation shows that $h^2\big(\Phi,\cG(t\eta)\big)=0$ too.

Finally $\gamma_1=0$, whence ${\cG}^\vee\cong\cG$, thus Serre's duality also yields $h^3\big(\Phi,\cG(t\eta)\big)=0$. We conclude that $\cG$ is an aCM bundle and it is trivial to check that $\cG$ is initialized.

If $\cG$ were decomposable, then $\cG\cong\cM\oplus \cM^{-1}$ because $\gamma_1=0$. Thus $\eta_2^2=\gamma_2=-{c}_1(\cM)^2$. But if $c_1(\cM)=m_1\eta_1+m_2\eta_2$, then $c_1(\cM)^2=m_1^2\eta_1^2+m_2^2\eta_2^2+2m_1m_2\eta_1\eta_2$ which cannot coincide with $-\eta_2^2$.

\begin{theorem}
\label{tplane}
Let $\cG$ be an indecomposable, initialized, aCM bundle of rank $2$ on $\Phi$ with $\gamma_1=0$. Then $\gamma_2=\eta_2^2$  up to permutation of the $\eta_i$'s.

Conversely, there exists an indecomposable, initialized, aCM bundle $\cG$ of rank $2$ on $\Phi$ with Chern classes $\gamma_1=0$ and $\gamma_2=\eta_2^2$.

Moreover the zero--locus of a general section of $\cG$ is a plane and each plane on $\Phi$ can be obtained in such a way.
\end{theorem}

An analogous statement holds for $F$.

\begin{theorem}
\label{tline}
Let $\cE$ be an indecomposable, initialized, aCM bundle of rank $2$ on $F$ with $c_1=0$. Then $c_2=h_2^2$  up to permutation of the $h_i$'s.

Conversely, there exists an indecomposable, initialized, aCM bundle $\cE$ of rank $2$ on $F$ with Chern classes $c_1=0$ and $c_2=h_2^2$.

Moreover the zero--locus of a general section of ${\cE}$ is a line and each line on ${F}$ can be obtained in such a way.
\end{theorem}
\begin{proof}
The only change is in the proof of the indecomposability of $\cE$. Indeed in this case $h_1h_2=h_1^2+h_2^2$, thus we have to verify that $(m_1^2+2m_1m_2){h}_1^2+(2m_1m_2+m_2^2){h}_2^2=-h_2^2$ cannot hold, which is a trivial computation.
\end{proof}

Besides the above general construction via codimension $2$ subschemes of $F$, we can also obtain examples of initialized, indecomposable aCM bundles with $c_1=0$ (and $\gamma_1=0$) by means of an interesting alternative construction via suitable pull--backs.

\begin{example}
\label{eLinePlane}
Let $\cE$ be any bundle defined by the exact sequence
$$
0\longrightarrow\cO_{F}(-h_2)\longrightarrow\cO_{F}\oplus p_2^*\Omega_{\p2}^1(1)\longrightarrow\cE\longrightarrow0
$$
(e.g. $\cE$ is the pullback on $F$ via $p_2$ of the restriction to a linear $\p2$ of a null correlation bundle on $\p3$). 
Hence $\cE$ is initialized, indecomposable and aCM on $F$ with $c_1=0$ and $c_2=h_2^2$.

Obviously a similar construction can be repeated on $\Phi$ obtaining an initialized, indecomposable, aCM bundle with $\gamma_1=0$ and $\gamma_2=\eta_2^2$.
\end{example}

Now we turn our attention to the case $c_1=2h$.  We first characterize such bundles in terms of the zero--locus of their general sections.

Assume that an indecomposable, initialized, aCM bundle $\cE$ of rank $2$ with $c_1=2h$ exists on $F$. Proposition \ref{pUpper6} guarantees that $\cE$ is Ulrich and that the zero--locus  $E:=(s)_0$ of a general section $s\in H^0\big(F,\cE\big)$ is a smooth curve inside $F$. 

Using Sequence \eqref{seqIdeal} we obtain that $h^1\big(F,\cI_{E\vert F}\big)=h^1\big(F,\cI_{E\vert F}(h)\big)=h^0\big(F,\cI_{E\vert F}(h)\big)=0$, whence both $h^0\big(E,\cO_{E}\big)=1$ and $E$ is linearly normal and non--degenerate (use Sequence \eqref{seqStandard}). Moreover, combining Lemma \ref{lInvariantE} and Formulas \eqref{hc_2}, \eqref{c_1c_2}, we obtain
$\deg(E)=hc_2=8$, $p_a(E)=1$.  We conclude that we can assume $E$ to be a non--degenerate, irreducible, smooth, elliptic curve of degree $8$, i.e. an elliptic normal curve.
 
Conversely, let $E$ be an elliptic normal curve on $F$. We know by adjunction that $\cO_E\cong\omega_E\cong\det(\cN_{E\vert F})\otimes\cO_F(-2h)$. The invertible sheaf $\cO_F(2h)$ satisfies the vanishing of the Theorem \ref{tSerre}, thus $E$ is the zero locus of a section $s$ of a vector bundle $\cE$ of rank $2$ on $F$ with $c_1=2h$ and the class of $E$ in $A^2(F)$ is $c_2$. 

Computing the cohomology of Sequence \ref{seqIdeal} twisted by $\cO_F(-h)$, we infer that $\cE$ is initialized, being $E$ non--degenerate. We now prove that it is also aCM.

Due to Proposition 1.1 and Corollary 2.2 of \cite{C--G--N}, $E$ is aCM, then $h^1\big(F,\cI_{E\vert\p7}(t)\big)=0$. The threefold $F$ is aCM, whence $h^2\big(F,\cI_{F\vert\p7}(t)\big)=0$. Taking the cohomology of sequence
$$
0\longrightarrow \cI_{F\vert\p7}\longrightarrow \cI_{E\vert\p7}\longrightarrow \cI_{E\vert F}\longrightarrow 0,
$$
it also follows that $h^1\big(F,\cI_{E\vert F}(th)\big)=0$. The twisted cohomology of the Sequence \eqref{seqIdeal} yields $h^1\big(F,\cE(th)\big)=0$. Such a vanishing also implies  $h^2\big(F,\cE(th)\big)=0$ by Serre's duality. We conclude that $\cE$ is aCM. Thanks to Proposition \ref{pUpper6} we also know that $\cE$ is Ulrich. Thus elliptic normal curves on $F$ correspond to Ulrich bundles on $F$ with $c_1=2h$. 

We deal with the possible values of $c_2$. Notice that the linear system $\vert h_i\vert$ on $F$ has dimension $2$ (see Proposition \ref{pLineBundle}). Let $D_i\in \vert h_i\vert$ be general. The cohomology of the exact sequence 
$$
0\longrightarrow \cO_F(h-h_i)\longrightarrow\cO_F(h)\longrightarrow\cO_{D_i}(h)\longrightarrow0
$$
yields $h^0\big(D_i,\cO_{D_i}(h)\big)=5$, thus $D_i$ spans a space of dimension $4$ in $\p7$. Since $E$ is non--degenerate, we know that the restriction to $E$ of $\vert h_i\vert$ has dimension at least $2$. Since $E$ is elliptic, it follows that its degree, which is $\beta_i$, is greater than $3$, thanks to Riemann--Roch theorem on the curve $E$. We conclude that the possible cases for $(\beta_1,\beta_2)$ are $(3,5)$, $(4,4)$ up to permutations, i.e. the bundle $\cE$ has $c_2$ equal to either  $3h_2^2+5h_1^2$, or $4h_1^1+4h_1^2$.

Both the cases actually occur and we will prove such an assertion in the two following examples.

\begin{example}
\label{e44}
Proposition \ref{pLineBundle} implies $\Ext^1_F\big(\cO_F(2h_1),\cO_F(2h_2)\big)\cong H^1\big(\cO_F(-2h_1+2h_2)\big)\cong k^{\oplus3}$.
In particular we obtain a family of indecomposable rank two Ulrich bundles on $F$ given by the non--trivial extensions
$$
0\longrightarrow\cO_F(2h_2)\longrightarrow\cE\longrightarrow\cO_F(2h_1)\longrightarrow 0.
$$
We have $c_1(\cE)=2h$ and $c_2(\cE)=4h_1h_2=4h_2^2+4h_1^2$.

We notice that the above construction cannot be extended on $\Phi$: indeed K\"unneth formulas imply that $h^1\big(\Phi,\mathcal L\big)=0$ for each line bundle $\mathcal L\in\Pic(\Phi)$.
\end{example}

\begin{example}
\label{e35}
As pointed out in \cite{C--M--P}, Notation 3.8, a non--empty family $\cM_\Phi$ of initialized, indecomposable, Ulrich bundles of rank $2$ is defined inside the moduli space of stable vector bundles of rank $2$ on $\Phi$ with Chern classes $\gamma_1=2\eta$, $\gamma_2=\eta_2^2+3\eta_1^2+2\eta_1\eta_2$. The elements of $\cM_\Phi$ fit into an exact sequence
\begin{equation}
\label{cmp}
0\longrightarrow\cG\to \cO_{\Phi}(2\eta_1+\eta_2)^{\oplus4}\longrightarrow \cO_{\Phi}(3\eta_1+\eta_2)^{\oplus2}\longrightarrow 0.
\end{equation}

Thanks to Theorem \ref{tVectorBundle} $\cE:=\cG\otimes\cO_F$ is an initialized, indecomposable Ulrich bundle on $F$: moreover $c_1=2h$ and $c_2=3h_2^2+5h_1^2$. It follows the existence of a non--empty family $\cM_F$ of initialized, indecomposable Ulrich bundles of rank $2$ inside the moduli space of stable vector bundles of rank $2$ on $F$ with such Chern classes. The restriction map $\cM_\Phi\to\cM_F$ is thus a morphism, which is surjective by construction.

\begin{claim}
The above restriction map is actually birational.
\end{claim}

To prove the above claim, it suffices to check that such a map is injective. To this purpose we first prove that if $\cG$, $\cG'$ are bundles in $\cM_\Phi$, then each morphism $u\colon \cG\otimes\cO_F\to \cG'\otimes\cO_F$ can be lifted to a morphism $v\colon \cG\to \cG'$, i.e. that the natural map $H^0\big(\Phi,\cG^\vee\otimes\cG'\big)\to H^0\big(F,\cG^\vee\otimes \cG'\otimes\cO_F\big)$ is surjective. Taking into account the cohomology standard restriction sequence of $\Phi$
$$
0\longrightarrow \cG^\vee\otimes\cG'(-\eta)\longrightarrow \cG^\vee\otimes\cG'\longrightarrow \cG^\vee\otimes \cG'\otimes\cO_F\longrightarrow 0
$$
it suffices to check that $h^1\big(\Phi,\cG^\vee\otimes\cG'(-\eta)\big)=0$.

Let us consider the dual of Sequence \eqref{cmp} tensored by $\cG'(-\eta)$, i.e.
$$
0\longrightarrow \cG'(-4\eta_1-2\eta_2)^{\oplus2}\longrightarrow \cG'(-3\eta_1-2\eta_2)^{\oplus4}\longrightarrow \cG^\vee\otimes\cG'(-\eta)\longrightarrow 0.
$$
It follows from Sequence \eqref{cmp} for $\cG'$ that $h^2\big(\Phi,\cG'(-4\eta_1-2\eta_2)\big)= h^1\big(\Phi,\cG'(-3\eta_1-2\eta_2)\big)=0$, we obtain $h^1\big(\Phi,\cG^\vee\otimes\cG'(-\eta)\big)=0$. 

Now assume that $u$ is an isomorphism. Thus 
$$
\det(v)\in H^0\big(\Phi,\det(\cG^\vee)\otimes\det(\cG')\big)\cong H^0\big(\Phi,\cO_\Phi\big)\cong k,
$$
because $\cG$ and $\cG'$ have the same first Chern class. Since the restriction to $F$ of $\det(v)$ is exactly $\det(u)$ which must be non--zero, being $u$ an isomorphism, it follows that $\det(v)$ is non--zero too, hence $v$ is an isomorphism.
\end{example}

\begin{theorem}
\label{tElliptic}
Let $\cE$ be an indecomposable, initialized, aCM bundle of rank $2$ on $F$ with $c_1=2h$. Then $c_2$ is either $3{h}_2^2+5{h}_2^2$, or $4{h}_2^2+4{h}_1^2$  up to permutation of the $h_i$'s.

Conversely, for each such a $c_2$, there exists an indecomposable, initialized, aCM bundle $\cE$ of rank $2$ on $F$ with Chern classes $c_1=2h$ and $c_2$.

Moreover each such a bundle $\cE$ is Ulrich, the zero--locus of a general section of ${\cE}$ is an elliptic normal curve and each elliptic normal curve on $F$ can be obtained in this way unless it is the complete intersection of two divisors in $\vert 2{h}_2\vert$ and $\vert 2{h}_1\vert$ (such a case can occur only if ${c}_2=4{h}_2^2+4{h}_1^2=4{h}_1{h}_2$).
\end{theorem}
\begin{proof}
Examples \ref{e44} and \ref{e35} prove the existence of the bundles. The remaining part of the proof follows from the discussion after Example \ref{eLinePlane}.
\end{proof}

We will now focus our attention on the analogous problem of classifying indecomposable, initialized, aCM bundles ${\cG}$ of rank $2$ on $\Phi$ with $\gamma_1=2\eta$ and $\gamma_2=\mu_1\eta_2^2+\mu_2\eta_1^2+\mu_3\eta_1\eta_2$. As pointed out at the beginning of this section, $\cG\otimes\cO_F$ is an indecomposable, initialized, aCM bundle  of rank $2$ on $F$ with $c_1=2h$ and $c_2=(\mu_1+\mu_3)h_2^2+(\mu_2+\mu_3)h_1^2$. Hence $(\mu_1+\mu_3,\mu_2+\mu_3)$ is either $(3,5)$ or $(4,4)$ by Theorem \ref{tElliptic}.

\begin{theorem}
\label{tdelPezzo}
Let $\cG$ be an indecomposable, initialized, aCM bundle of rank $2$ on $\Phi$ with $\gamma_1=2\eta$. Then $\gamma_2=\eta_2^2+3\eta_1^2+2\eta_1\eta_2$  up to permutation of the $h_i$'s.

Conversely, there exists an indecomposable, initialized, aCM bundle $\cG$ of rank $2$ on $\Phi$ with Chern classes $\gamma_1=2\eta$ and $\gamma_2=\eta_2^2+3\eta_1^2+2\eta_1\eta_2$.

Moreover each such a bundle $\cG$ is Ulrich, the zero--locus of a general section of ${\cG}$ is a del Pezzo surface of degree $8$ isomorphic to the blow up $\mathbb F_1$ of $\p2$ at a point and each del Pezzo surface of degree $8$ on $\Phi$ can be obtained in this way unless it is the complete intersection of two divisors in $\vert 2{\eta}_2\vert$ and $\vert 2{\eta}_1\vert$ (such a case can occur only if ${\gamma}_2=4{\eta}_1{\eta}_2$).
\end{theorem}
\begin{proof}
The claimed existence follows from Example \ref{e35}. 

In order to complete the proof of the statement, we preliminary observe that if $\gamma_1=2\eta$, then $\cG$ is Ulrich, hence its general section vanishes on a smooth surface $\Sigma$ whose class in $A^2(\Phi)$ is $\gamma_2=\mu_1\eta_2^2+\mu_2\eta_1^2+\mu_3\eta_1\eta_2$. Recall that $(\mu_1+\mu_3,\mu_2+\mu_3)$ can be assumed either $(3,5)$, or $(4,4)$.

The same argument used for elliptic curves on $F$ shows that $\Sigma\subseteq\p8$ is non--degenerate. Restricting Sequence \eqref{seqIdeal} for $\Sigma$ (which is a Koszul complex) to $\Sigma$ itself we obtain that $\cG\otimes\cO_\Sigma\cong\cN_{\Sigma\vert\Phi}$. Adjunction formula finally yields $\omega_\Sigma\cong\cO_\Sigma(-\eta)$. We conclude that $\Sigma$ is a del Pezzo surface embedded anticanonically in $\p8$, thus its degree is $8$. Looking at the classification of del Pezzo surfaces we know that $\Sigma$ is isomorphic to either $Q:=\p1\times\p1$, or $\Bbb F_1:=\bP(\cO_{\p1}\oplus\cO_{\p1}(-1))$. Conversely, again imitating the argument used for elliptic curves on $F$, recalling that each del Pezzo surface and each del Pezzo sextic fourfold in $\p8$ are aCM, it is also possible to prove that each del Pezzo surface $\Sigma\subseteq\Phi$ is the zero locus of an Ulrich bundle $\cG$ of rank $2$ on $\Phi$ with $\gamma_1=2h$. Thus, in order to prove the statement it suffices to deal with the possible embeddings $\Sigma\subseteq\Phi$. 

We start from the case $\Sigma\cong\Bbb F_1$. Recall that $\Pic(\Bbb F_1)$ is freely generated by divisors $\ell$ and $m$ such that $\ell^2=-1$, $m^2=0$ and $\ell m=1$.  Let $\Sigma\subseteq\Phi$ be any smooth del Pezzo octic surface isomorphic to $\Bbb F_1$ in the class $\gamma_2$.

Each morphism $\varphi\colon \Bbb F_1\to\Phi$ induces, by compositions with the two natural projections $\Phi\to\p2$, maps $\varphi_i\colon \Bbb F_1\to\p2$. In particular we know that $\varphi^*\cO_{\Phi}(\eta_i)\cong \varphi_i^*\cO_{\p2}(1)\cong \cO_{\Bbb F_1}(a_i\ell+b_i m)$ for suitable integers $a_i, b_i$.  We thus have $\varphi^*\cO_{\Phi}(\eta)\cong \cO_{\Bbb F_1}((a_1+a_2)\ell+(b_1+b_2)m)$.

Since $\cO_{\Bbb F_1}(a_i\ell+b_i m)$ is effective, it follows that $a_i$ and $b_i$ are non--negative. Moreover the embedding $\Sigma\subseteq\p8$ is anticanonical, thus $a_1+a_2=2$ and $b_1+b_2=3$. We have $\mu_i=\eta_i^2\Sigma=(a_i\ell+b_im)^2=a_i(2b_i-a_i)$, $i=1,2$. 

Let us first examine the case $\mu_1+\mu_3=\mu_2+\mu_3=4$. If $a_1=a_2=1$, then $2b_i-1=\mu_i$ must be odd. Moreover 
$$
2\mu_3=8-\mu_1-\mu_2=10-2b_1-2b_2=10-6=4.
$$
We conclude that $\mu_1=\mu_2=\mu_3=2$, a contradiction. If $a_1=0$ and $a_2=2$, then $\mu_1=0$, whence $\mu_2=0$ and $\mu_3=4$. The equality $2(2b_2-2)=\mu_2=0$ thus yields $b_2=1$. The morphism $\varphi_2$ should be thus defined by three sections of $\cO_{\Bbb F_1}(2\ell+m)$ without common zeros, which is absurd, because all such sections vanish identically on $\ell$. A similar argument holds in the case $a_1=2$ and $a_2=0$.

Now we examine the second case $\mu_1+\mu_3=3$, $\mu_2+\mu_3=5$. Taking into account of the listed constraints, arguing as above it is not difficult to check that $a_1=a_2=b_1=1$ and $b_2=2$. It follows that $\mu_1=1$, $\mu_2=3$, $\mu_3=2$, i.e. $\gamma_2=\eta_2^2+3\eta_1^2+2\eta_1\eta_2$. 

Notice that each bundle $\cG$ with such a  $\gamma_2$ is, necessarily, indecomposable. Indeed, if $\cG\cong{\mathcal L}_1\oplus{\mathcal L}_2$ for some ${\mathcal L}_i\in\Pic(\Phi)$, then ${\mathcal L}_i$ is aCM. Either both the ${\mathcal L}_i$ are initialized or one of them is initialized and the other one has no sections. Thanks to the equalities
$$
  h^0\big(\Phi,{\mathcal L}_1\big)+h^0\big(\Phi,{\mathcal L}_2\big)=h^0\big(\Phi,\cG\big)=12\qquad
  c_1({\mathcal L}_1)+c_1({\mathcal L}_2)=2\eta
$$
and to Proposition \ref{pLineBundle}, it is easy to check that $\cG\cong\cO_{\Phi}(2\eta_2)\oplus\cO_{\Phi}(2\eta_1)$ necessarily. In particular we should have $\gamma_2=4\eta_1\eta_2$, a contradiction.  

We finally examine the case of the smooth quadric $Q$. Recall that $\Pic(Q)$ is freely generated by divisors $\ell$ and $m$ such that $\ell^2=m^2=0$ and $\ell m=1$. 

As in the previous case we have two morphisms $\varphi_i\colon Q\to\p2$ with $\varphi^*\cO_{\Phi}(\eta_i)\cong \varphi_i^*\cO_{\p2}(1)\cong \cO_{Q}(a_i\ell+b_i m)$ for suitable integers $a_i, b_i$.  Thus again $\varphi^*\cO_{\Phi}(\eta)\cong \cO_{Q}((a_1+a_2)\ell+(b_1+b_2)m)$.

In this case the list of constraints is as follows: $a_i,b_i\ge 0$, $a_1+a_2=b_1+b_2=2$, $\mu_i=\eta_i^2\Sigma=(a_i\ell+b_im)^2=2a_ib_i$, $i=1,2$. 

Let us start from the case $\mu_1+\mu_3=3$, $\mu_2+\mu_3=5$. If $a_1=a_2=1$, then
$$
2\mu_3=8-\mu_1-\mu_2=8-2b_1-2b_2=8-4=4,
$$
thus $2b_1=\mu_1=1$, a contradiction. If $a_1=0$ and $a_2=2$, then $\mu_1=0$, $\mu_2=2$ and $\mu_3=3$. In particular $b_1=b_2=1$, hence the morphism $\varphi_1$ should be defined by three sections of $\cO_{Q}(m)$. It would follow that its image should be contained in a line in $\p2$, thus $\Sigma$ would be contained in a divisor of the linear system $\vert\eta_1\vert$. We conclude that $\Sigma$ should be degenerate, a contradiction. The argument in the case $a_1=2$ and $a_2=0$ is similar.

Finally let us consider the case $\mu_1+\mu_3=\mu_2+\mu_3=4$. Arguing as in the previous cases, we deduce that either $a_1=b_2=0$ and $a_2=b_1=2$, or $a_2=b_1=0$ and $a_1=b_2=2$, or $a_1=b_1=a_2=b_2=1$. In the two former cases $\mu_1=\mu_2=0$ whence $\mu_3=4$, in the latter $\mu_1=\mu_2=2$, whence $\mu_3=2$.

Consider the case $\mu_1=\mu_2=0$, $\mu_3=4$. The image of the morphism $Q\to\p{d}$ associated to the complete linear system $\vert\cO_Q(d\ell)\vert$ (or $\vert\cO_Q(dm)\vert$) is a rational normal curve for each $d\ge1$. In particular, we notice that $D_i:=\varphi_i(Q)\subseteq \p2$ is a conic, whence $\varphi_i^{-1}(D_i)\in\vert 2\eta_i\vert$. Since  $\Sigma\subseteq\varphi_i^{-1}(D_1)\cap\varphi_i^{-1}(D_2)$ and both the schemes are in the class $4\eta_1\eta_2$ in $A^2(\Phi)$, it follows that $\Sigma$ is complete intersection inside $\Phi$. Since $h^2\big(\Phi,\cO_{\Phi}(-2\eta)\big)=0$, it follows from Theorem \ref{tSerre} the existence of two exact sequences
\begin{gather*}
0\longrightarrow \cO_{\Phi}\longrightarrow \cG\longrightarrow \cI_{\Sigma\vert \Phi}(2\eta)\longrightarrow 0,\\
0\longrightarrow \cO_{\Phi}\longrightarrow \cO_{\Phi}(2\eta_1)\oplus\cO_{\Phi}(2\eta_2)\longrightarrow \cI_{\Sigma\vert \Phi}(2\eta)\longrightarrow 0
\end{gather*}
which must coincide because $h^1\big(\Phi,\cO_{\Phi}(-2\eta)\big)=0$ (see again Theorem \ref{tSerre}). We conclude that
$\cG\cong \cO_{\Phi}(2\eta_1)\oplus\cO_{\Phi}(2\eta_2)$. 

We conclude with the case $\mu_1=\mu_2=\mu_3=2$. In this case the maps $\varphi_i$ are the restriction of projections $\psi_i\colon \p3\dashrightarrow\p2$ from a point. Let $\psi\colon \p3\dashrightarrow\Phi$ be their product rational map. By construction $\varphi=\psi_{\vert Q}$. 

By definition of product map, the equations defining $\psi$ are
$$
(x_0,x_1,x_2,x_3,x_4,x_5,x_6,x_7,x_8)=(y_0y_3,y_0y_4,y_0y_5,y_1y_3,y_1y_4,y_1y_5,y_2y_3,y_2y_4,y_2y_5)
$$
where $y_0,y_1,y_2$ and $y_3,y_4,y_5$ are two triples of independent sections of $H^0\big(\p3,\cO_{\p3}(1)\big)$. Let $V\subseteq H^0\big(\p3,\cO_{\p3}(1)\big)$ the subspace generated by all these sections: we distinguish the two cases $\dim(V)=3,4$. 

In the last case we can assume $y_4=\sum_{i=0}^3u_iy_i$ and $y_5=\sum_{j=0}^3v_jy_j$. Since $y_3,y_4,y_5$ are independent, it follows that at least one of the $u_i$'s, $i=0,1,2$ is non--zero.
It is easy to check that $\im(\psi)$ is contained in the hyperplane with equation
\begin{align*}
(u_0v_3-u_3v_0)x_0&+v_0x_1+(u_1v_3-u_3v_1)x_3+v_1x_4+\\
&+(u_2v_3-u_3v_2)x_6+v_2x_7=u_0x_2+u_1x_5+u_2x_8.
\end{align*}

In the first case we can write $y_3=\sum_{i=0}^2u_iy_i$, $y_4=\sum_{j=0}^2v_jy_j$, $y_5=\sum_{j=0}^2w_jy_j$. Since $y_3,y_4,y_5$ are independent, it follows that at least one of the $u_i$'s is non--zero. Again one checks easily that the hyperplane with equation
$$
w_0x_0+w_1x_3+w_2x_6=u_0x_2+u_1x_5+u_2x_8
$$
contains $\im(\psi)$.

Thus $\Sigma\subseteq\im(\psi)\subseteq\p8$ is contained in a hyperplane too. In particular, $\Sigma$ is not a del Pezzo octic, because it is degenerate. We conclude that also this second case cannot occur.
\end{proof}

\section{The intermediate cases}
\label{sIntermediate}
In this section we will classify the indecomposable, initialized, aCM bundles of rank $2$ on $F$ and $\Phi$ whose first Chern class is not extremal, i.e. is neither $0$, nor $2h$, nor $2\eta$. 

We will show that there there are only two possible bundles (both on $F$ and on $\Phi$) up to permuting the generators of the Picard group. Such bundles are described in the two following examples.

\begin{example}
Consider the indecomposable bundle $\cG:=\pi_2^*\Omega_{\p2}^1(2\eta_2)$. Using K\"unneth formulas it is easy to check that $\cG$ is initialized and aCM. Moreover its Chern classes are $\gamma_1=\eta_2$ and $\gamma_2=\eta_2^2$.
\end{example}

\begin{example}
\label{fond2}
We have $\Ext^1_F(\cO_F(h_1),\cO_F(-h_1+h_2))\cong H^1(F, \cO_F(-2h_1+h_2))\cong k$ (see Proposition \ref{pLineBundle}). It follows the existence of a unique indecomposable extension
\begin{equation}
\label{seqOmega}
0\longrightarrow\cO_F(-h_1+h_2)\longrightarrow \cE\longrightarrow\cO_F(-h_1)\longrightarrow 0.
\end{equation}
It is immediate to check that $\cE$ is initialized, indecomposable and aCM. Its Chern classes are  $c_1=h_2$ and $c_2=h_2^2$.

\begin{claim}
$\cE\cong p_2^*\Omega_{\p2}^1(2h_2)$.
\end{claim}

It suffices to prove the existence of another extension like Sequence \eqref{seqOmega} with $p_2^*\Omega_{\p2}^1(2h_2)$ in the middle. In fact, let $\cG:=\pi_2^*\Omega_{\p2}^1(2\eta_2)$ and consider the standard restriction sequence
$$
0\longrightarrow \cG(-2\eta_2)\longrightarrow \cG(\eta_1-\eta_2) \longrightarrow \cG\otimes\cO_F(h_1-h_2)\longrightarrow 0.
$$
K\"unneth formula yields $h^1\big(\Phi,\cG(-2\eta_2)\big)=1$ and $h^0\big(\Phi,\cG(\eta_1-\eta_2)\big)=h^1\big(\Phi,\cG(\eta_1-\eta_2)\big)=0$, whence $h^0\big(F,\cG\otimes\cO_F(h_1-h_2)\big)=1$.  Let $s\in H^0\big(F,\cG\otimes\cO_F(h_1-h_2)\big)$ be a general section and consider the corresponding Sequence \eqref{seqIdeal}
$$
0\longrightarrow \cO_F(D)\longrightarrow \cG\otimes\cO_F(h_1-h_2)\longrightarrow \cI_E(2h_1-h_2-D)\longrightarrow 0,
$$
where $D$ is an effective divisor and $E$ has pure codimension $2$ in $F$ or it is empty.

On the one hand $h^0\big(F,\cG\otimes\cO_F(h_1-h_2)\big)=1$, thus $h^0\big(F,\cO_F(D)\big)\le1$. This last inequality forces the vanishing $D=0$, due to Theorem \ref{tVectorBundle}. On the other hand an easy computation shows that $c_2(\cG\otimes\cO_F(h_1-h_2))=0$. It follows that $\cG\otimes\cO_F(h_1-h_2)$ fits into the exact sequence
$$
0\longrightarrow \cO_F\longrightarrow \cG\otimes\cO_F(h_1-h_2)\longrightarrow \cO_F(2h_1-h_2)\longrightarrow 0
$$
which is the claimed extension twisted by $\cO_F(-h_1+h_2)$.
\end{example}

As claimed at the beginning of this section, we will show in the following that the two bundles described above are the only possible indecomposable, initialized aCM bundles of rank $2$ on $F$ and $\Phi$.

Inequality \eqref{positivity3} with $D=0$ and Formulas \eqref{c_1c_2}, \eqref{hc_2}  yield the following list (up to permutations) for the invariants associated to $\cE$ and $E$.
$$
\begin{array}{|c|c|c|c|c|}           \hline
  \text{Case} &{ (\alpha_1,\alpha_2) } & {(\beta_1,\beta_2) } & { \deg(E) } & { p_a(E) } \\ \hline
  \text{L}&(0,1) & (1,0) & 1& 0\\  \hline
  \text{M}&(0,2) & (1,0) & 1 & 0\\  \hline
  \text{N}&(1,1) & (2,0) & 2 & 0\\  \hline
  \text{P}&(1,1) & (1,1) & 2 & 0\\  \hline
  \text{Q}&(1,2) & (2,2) & 4 & 0\\  \hline
\end{array}
$$

Let $\cE$ have the invariants in the list above. Then $\cE^\vee(th)$ is an indecomposable aCM bundle too. We take $t$ in such a way that it is also initialized. Propositions \ref{pLower6} and \ref{pUpper6} imply $0\le 2t-\alpha_i\le2$, i.e. $t=1$. We have
\begin{gather*}
  c_1(\cE^\vee(h))=(2-\alpha_1)h_1+(2-\alpha_2)h_2,\\
  c_2(\cE^\vee(h))=(\beta_2-2\alpha_1-\alpha_2+3)h_1^2+(\beta_1-\alpha_1-2\alpha_2+3)h_2^2.
 \end{gather*}
 It follows that case L occurs if and only if case Q occurs too.

\subsection{The cases L and Q}
In case L we know by the table that the cycle of $E$ in $A^2(F)$ must be  $h_2^2$.

Let $E\subseteq F$ be such a line. We have $\omega_E\cong\cO_E(-2h)$, hence $\det(\cN_{E\vert F})\cong\cO_E\cong\cO_F(h_2)\otimes\cO_E$. Taking into account that $h^2\big(F,\cO_F(-h_2)\big)=0$ (see Proposition \ref{pLineBundle}), we thus know the existence of an exact sequence
\begin{equation}
\label{seqLine}
0\longrightarrow \cO_F\longrightarrow \cE\longrightarrow\cI_{E\vert F}(h_2)\longrightarrow0.
\end{equation}
where $\cE$ is a vector bundle of rank $2$ on $F$ (see Theorem \ref{tSerre}). It is immediate to check that $\cE$ is initialized.

The cohomology of Sequence \eqref{seqStandard} twisted by $\cO_F(h_1)$ gives rise to the exact sequence
$$
  H^0\big(F,\cO_F(h_1)\big)\longrightarrow H^0\big(E,\cO_E(h_1)\big)\longrightarrow H^1\big(F,\cI_{E\vert F}(h_1)\big)\longrightarrow0.
$$
The multiplication by $h_1$ gives rise to the commutative diagram
$$
\begin{CD}
  H^0\big(F,\cO_{F}\big)@>\sim>> H^0\big(E,\cO_E\big)\\
  @VVV @V\psi VV\\
  H^0\big(F,\cO_{F}(h_1)\big)@>>> H^0\big(E,\cO_E(h_1)\big)\\
\end{CD}
$$
The horizontal upper arrow is an isomorphism, because E is connected. Moreover the cycle of $E$ is $h_2^2$, thus $\psi$, which is the multiplication by $h_1$, is an isomorphism too. It follows that the map below is surjective, whence $h^1\big(F,\cI_{E\vert F}(h_1)\big)=0$. 

Since $h^0\big(E,\cO_{E}(h_1)\big)=2$ and $h^0\big(F,\cO_F(h_1)\big)=3$, it follows that the cohomology of Sequence \eqref{seqStandard} twisted by $\cO_F(h_1)$ yields $h^0\big(F,\cI_{E\vert F}(h_1)\big)=1$. The cohomology of Sequence \eqref{seqLine} twisted by $\cO_F(h_1-h_2)$ thus yields  $h^0\big(F,\cE(h_1-h_2)\big)=1$. 

Let us consider the zero locus $Z$ of a general section of $\cE(h_1-h_2)$. As usual we decompose $Z=C\cup\Delta$ where $C$ has pure codimension $2$ (or it is empty) and $\Delta$ is an effective divisor. The corresponding Koszul complex is
$$
0\longrightarrow \cO_F(\Delta)\longrightarrow \cE(h_1-h_2)\longrightarrow\cI_{C\vert F}(2h_1-h_2-\Delta)\longrightarrow0.
$$

We deduce that $h^0\big(F,\cO_{F}(\Delta)\big)\le1$, hence Proposition \ref{pLineBundle} and the effectiveness of $\Delta$ yield $\Delta=0$. Thus $C$ is in the class of $c_2(\cE(h_1-h_2))=0$, hence it is empty. We conclude that the above sequence twisted by $\cO_F(-h_1+h_2)$ is
$$
0\longrightarrow \cO_F(-h_1+h_2)\longrightarrow \cE\longrightarrow\cO_F(h_1)\longrightarrow 0.
$$ 
which is the extension of Example \ref{fond2} because $\Ext^1_F\big(\cO_{F}(2h_1-h_2),\cO_F\big)\cong k$. In particular the bundle $\cE$ is indecomposable, initialized and aCM, thus both the cases L and Q are admissible.

\subsection{The case M}
In case M we know by the table that $E\subseteq F$ is a line, whose class in $ A^2(F)$ is $h_2^2$ and we know that $\det({\cN_{E\vert F}})\cong\cO_F(2h_2)\otimes\cO_E$. Since $h^2\big(F,\cO_F(-2h_2)\big)=0$, it follows the existence of an exact sequence
$$
0\longrightarrow \cO_F(-2h_2)\longrightarrow \cE^\vee\longrightarrow\cI_{E\vert F}\longrightarrow0.
$$
where $\cE$ is a vector bundle of rank $2$ on $F$. On the other hand we also have the exact sequence
$$
0\longrightarrow \cO_F(-2h_2)\longrightarrow \cO_F(-h_2)^{\oplus2}\longrightarrow\cI_{E\vert F}\longrightarrow0.
$$

Since $h^1\big(F,\cO_F(-2h_2)\big)=0$ (Proposition \ref{pLineBundle}), we conclude that $\cE\cong\cO_F(h_2)^{\oplus2}$ thanks to Theorem \ref{tSerre}.

In particular case M is not admissible in the sense that the sheaf $\cE$ exists, but it is decomposable.

\subsection{The cases N and P}
In both these cases $E$ is a curve of degree $2$ of genus $0$. The curve $E$ is either irreducible or the union of two concurrent lines, because $p_a(E)=0$, or a double line.

In case P only the first two cases are possible because the class of $E$ is $h_2^2+h_1^2$, thus $E$ is reduced. In this case $\omega_E\cong\cO_E(-h)$, whence $\det(\cN_{E\vert F})\cong\cO_E(h)$. Since $h^2\big(F,\cO_F(h)\big)=0$, we thus know the existence of an exact sequence
$$
0\longrightarrow \cO_F\longrightarrow \cE\longrightarrow\cI_{E\vert F}(h)\longrightarrow0.
$$
where $\cE$ is a vector bundle of rank $2$ on $F$. 

Since $E$ is either a conic or the union of two distinct concurrent lines and $h_2c_2=1$, it follows that $h^0\big(E,\cO_E(h_2)\big)=2$. Thus the cohomology of Sequence \ref{seqStandard} twisted by $\cO_F(h_2)$ implies on the one hand $h^0\big(F,\cI_{E\vert F}(h_2)\big)\ge1$. On the other hand $h^0\big(F,\cI_{E\vert F}(h_2)\big)<h^0\big(F,\cO_{F}(h_2)\big)$, because $\cO_F(h_2)$ is base point free,  whence $h^0\big(F,\cI_{E\vert F}(h_2)\big)\le2$.

We conclude that the cohomology of the above sequence twisted by $\cO_F(-h_1)$ and Corollary \ref{cLineBundle} thus imply $1\le h^0\big(F,\cE(-h_1)\big)\le2$.

Let us consider the zero locus $Z$ of a general section of $\cE(-h_1)$. Let $Z=C\cup\Delta$ where $C$ has pure codimension $2$ (or it is empty) and $\Delta$ is an effective divisor. The corresponding Koszul complex is
$$
0\longrightarrow \cO_F(\Delta)\longrightarrow \cE(-h_1)\longrightarrow\cI_{C\vert F}(-h_1+h_2-\Delta)\longrightarrow0.
$$
We deduce that $h^0\big(F,\cO_{F}(\Delta)\big)\le2$. Proposition \ref{pLineBundle} and the effectiveness of $\Delta$ yield $\Delta=0$. Thus $C$ is in the class of $c_2(\cE(-h_1))=0$. It follows that $\cE(-h_1)$ corresponds to an element of $\Ext^1_F\big(\cO_{F}(-h_1+h_2),\cO_F\big)= H^1\big(F,\cO_F(h_2-h_1)\big)=0$. We infer that $\cE=\cO_F(h_1)\oplus\cO_F(h_2)$.

Let us examine case N. Notice that two concurrent lines on $F$ have necessarily classes $h_1^2$ and $h_2^2$ in $A^2(F)$. Since the class of $E$ in $A^2(F)$ is $2h_2^2$, it follows that  $E$ is either irreducible or it is a double line. 

Let us examine the first case. Clearly there are distinct $U,V\in\vert h_2\vert$, $i=1,2$ such that $E\cap U\cap V\ne\emptyset$. Since $E U=E V=0$, it would follow that $E\subseteq U\cap V$, whence $2=2h_2^2h=\deg(E)\le \deg(U\cap V)=h_2^2h=1$, a contradiction. 

Thus $E$ is a double structure on a line $E_{red}$ whose cycle in $A^2(F)$ is $h_2^2$. The general theory of double structures (see e.g. \cite{B--E}) gives us an exact sequence of the form
$$
0\longrightarrow \cC_{E_{red}\vert E}\longrightarrow \cO_E\longrightarrow \cO_{E_{red}}\longrightarrow 0.
$$
The conormal sheaf $\cC_{E_{red}\vert E}$ is an invertible sheaf on $E_{red}\cong\p1$, thus $\cC_{E_{red}\vert E}\cong\cO_{\p1}(-a)$. Moreover we know that $a=1$, because $p_a(E)=0$ (see Section 2 of \cite{B--E}).

Recall that the Hilbert scheme $\Gamma$ of lines in $F$ has exactly two components $\Gamma_1$ and $\Gamma_2$, each of them isomorphic to $\p2$ (see Proposition 3.5.6 of \cite{I--P}). 

The lines with cycle $h_2^2$ form one of the two component, say $\Gamma_1$, and the universal family $S_1\to\Gamma_1$ is nothing else the natural projection $p_2\colon\bP(\Omega_{\p2}^1(2))\to\p2$. The variety $F$ does not contain planes, hence the universal family $S_1\to\Gamma_1$ maps surjectively on $F$ (see Proposition 3.3.5 of \cite{I--P}). Lemma 3.3.4 of \cite{I--P} implies that the line $U$ corresponding to the  general point in $\Gamma_1$ satisfies $\cC_{U\vert F}\cong\cN_{U\vert F}^\vee\cong\cO_{\p1}^{\oplus2}$. Clearly the induced action of the automorphism group of $F$ on $\Gamma_1$ is transitive, hence there is an automorphism of $F$ fixing the class $h_2^2$ and sending $E_{red}$ to $U$. We conclude that  $\cC_{E_{red}\vert F}\cong\cO_{\p1}^{\oplus2}$. 

The inclusion $E_{red}\subseteq E\subseteq F$ implies the existence of an epimorphism
$\cO_{\p1}^{\oplus2}\cong\cC_{E_{red}\vert F}\twoheadrightarrow \cC_{E_{red}\vert E}\cong\cO_{\p1}(-1)$, an absurd.

In particular both the cases P and N are not admissible: in case P the bundle $\cE$ exists but it is decomposable, in case N it does not exist at all.

We summarize the results of this section in the following result.
\begin{theorem}
\label{tIntermediate3}
Let $\cE$ be an indecomposable, initialized, aCM bundle of rank $2$ on $F$ and let $c_1=\alpha_1h_1+\alpha_2h_2$. Assume that $c_1$ is neither $0$, nor $2h$. Then  $(c_1,c_2)$ is either $(h_2, h_2^2)$, or $(h_1+2h_2, 2h_1^2+2h_2^2)$, up to permutations of the $h_i$'s.

Conversely, for each such a pair, there exists a unique indecomposable, initialized, aCM bundle $\cE$ of rank $2$ on $F$ with Chern classes $c_1$ and $c_2$. If $c_1=h_2$, then $\cE\cong p_2^*\Omega_{\p2}^1(2h_2)$. If $c_1=h_1+2h_2$, then $\cE\cong p_1^*\Omega_{\p2}^1(2h_1+h_2)$.

Moreover the zero--locus of a general section of $\cE$ is respectively a line, or a, possibly reducible, quartic with arithmetic genus $0$. Each line can be obtained in this way.
 \end{theorem}

 We conclude the section by dealing with initialized, indecomposable, aCM bundles $\cG$ of rank $2$ on $\Phi$. As usual let $\gamma_1=\alpha_1\eta_1+\alpha_2\eta_2$ and $\gamma_2=\mu_1\eta_2^2+\mu_2\eta_1^2+\mu_3\eta_1\eta_2$. As in the previous section we know $\cE=\cG\otimes\cO_F$ has second Chern class $c_2=(\mu_1+\mu_3)h_2^2+(\mu_2+\mu_3)h_1^2$. Thus Remark \ref{rRestriction}, Theorem \ref{tIntermediate3} and Inequalities \eqref{positivity4} imply that we have to deal with the existence of the following cases

 $$
\begin{array}{|c|c|c|c|c|}           \hline
  \text{Case} &{ (\alpha_1,\alpha_2) } & {(\mu_1,\mu_2,\mu_3) } \\ \hline
  \text{L}&(0,1) & (1,0,0) \\  \hline
\text{Q}&(1,2) & (1,1,1) \\  \hline
\text{Q}'&(1,2) & (2,2,0) \\  \hline
\text{Q}''&(1,2) & (0,0,2) \\  \hline
\end{array}
$$

As for $F$, if  $\cG$ is initialized, indecomposable aCM bundle, then the same holds true for ${\cG}^\vee(\eta)$. The computation of $c_2({\cG}^\vee(\eta))$ shows that cases Q$'$ and Q$''$ cannot occur, hence only the cases L and Q are admissible.

\begin{theorem}
\label{tIntermediate4}
Let $\cG$ be an indecomposable, initialized, aCM bundle of rank $2$ on $\Phi$ and let $c_1=\alpha_1\eta_1+\alpha_2\eta_2$. Assume that $\gamma_1$ is neither $0$, nor $2\eta$. Then  $(\gamma_1,\gamma_2)$ is either $(\eta_2, \eta_2^2)$, or $(\eta_1+2\eta_2, \eta_1^2+\eta_2^2+\eta_1\eta_2)$, up to permutations of the $\eta_i$'s.

Conversely, for each such a pair, there exists a unique indecomposable, initialized, aCM bundle $\cG$ of rank $2$ on $\Phi$ with Chern classes $\gamma_1$ and $\gamma_2$. If $\gamma_1=\eta_2$, then $\cG\cong \pi_2^*\Omega_{\p2}^1(2\eta_2)$. If $\gamma_1=\eta_1+2\eta_2$, then $\cG\cong \pi_1^*\Omega_{\p2}^1(2\eta_1+\eta_2)$.

Moreover the zero--locus of a general section of $\cG$ is respectively a plane, or a, possibly reducible, quartic surface. Each plane on $F$ can be obtained in this way.
 \end{theorem}
 \begin{proof}
 The discussion above and Theorem \ref{tVectorBundle} prove the first part of the statement. The existence of such kind of bundles and the fact that planes can be all obtained via such a construction can be proved with the same argument used for the threefold $F$.
 \end{proof}

\bigskip
\noindent
Gianfranco Casnati,\\
Dipartimento di Scienze Matematiche, Politecnico di Torino,\\
c.so Duca degli Abruzzi 24,\\
10129 Torino, Italy\\
e-mail: {\tt gianfranco.casnati@polito.it}

\bigskip
\noindent
Daniele Faenzi,\\
UFR des Sciences et des Techniques, Universit\'e de Bourgogne,\\
9 avenue Alain Savary -- BP 47870\\
21078 Dijon Cedex, France\\
e-mail: {\tt daniele.faenzi@u-bourgogne.fr}

\bigskip
\noindent
Francesco Malaspina,\\
Dipartimento di Scienze Matematiche, Politecnico di Torino,\\
c.so Duca degli Abruzzi 24,\\
10129 Torino, Italy\\
e-mail: {\tt francesco.malaspina@polito.it}

\end{document}